\documentclass{article}
\usepackage[utf8]{inputenc}
\usepackage{amsmath,amsthm, amssymb, amsfonts}
\usepackage{todonotes}
\usepackage{kantlipsum}
\usepackage{wrapfig}
\usepackage{paralist}
\usepackage{wasysym}
\usepackage{verbatim}

\usepackage{fullpage}
\usepackage{mathtools}
\usepackage[pdfencoding=auto]{hyperref}
\usepackage{graphicx}
\usepackage{bbm}
\usepackage{wrapfig}
\usepackage{todonotes}
\usepackage{mathrsfs}
\usepackage{enumerate}
\usepackage{dsfont}
\usepackage{units}
\usepackage[]{algorithm2e}
\usepackage{ltablex}
\usepackage{varwidth}
\usepackage{pgfplots}
\usepackage{subfigure}

\usepackage{xcolor}

\allowdisplaybreaks[1]
\usepackage{sistyle} 
\newcommand\bSI[1]{{\small[\SI{}{#1}]}}

\makeatletter
\newlength\unitwdth
\newlength\numwdth
\newlength\tdima
\newcommand\SIdescr[2]{%
    \setlength\tdima{\linewidth}%
    \addtolength\tdima{\@totalleftmargin}%
    \addtolength\tdima{-\dimen\@curtab}%
    \addtolength\tdima{-\unitwdth}%
    \addtolength\tdima{-\numwdth}%
    \parbox[t]{\tdima}{%
        #1
        \leaders\hbox{$\m@th\mkern \@dotsep mu\hbox{\tiny.}\mkern \@dotsep mu$}%
        \hfill
        \ifhmode\strut\fi
        \makebox[0pt][l]{%
            \makebox[\unitwdth][l]{}%
            \makebox[\numwdth][r]{#2}}}}
\makeatother

\newcommand{\sizefirst}{W_{\operatorname{fi}}} 
\newcommand{\sizelast}{W_{\operatorname{la}}} 
\newcommand{\Parallel}[1]{\mathrm{P}\left(#1\right)} 

\newcommand{\Z}{\mathbb{Z}}
\newcommand{\N}{\mathbb{N}}

\newcommand{\R}{\mathbb{R}}

\newcommand{\mani}{\mathcal{M}}

\newcommand{\NN}{\mathcal{NN}}
\newcommand{\RNN}{\mathcal{RNN}}

\newcommand{\eps}{\varepsilon}
\newcommand{\Realization}{\mathrm{R}}

\DeclareMathOperator*{\supp}{supp}

\newcommand{\lGamma}{\Gamma}

\newcommand{\sconc}{\odot}

\newtheorem{theorem}{Theorem}[section]
\newtheorem*{theorem*}{Theorem}
\newtheorem{remark}[theorem]{Remark}
\newtheorem{definition}[theorem]{Definition}
\newtheorem{proposition}[theorem]{Proposition}
\newtheorem{lemma}[theorem]{Lemma}

\newtheorem*{remark*}{Remark}
\newtheorem*{proposition*}{Proposition}

\numberwithin{equation}{section}

\DeclareMathOperator{\spann}{span \,}
\DeclareMathOperator{\suppp}{supp \,}
\DeclareMathOperator*{\argmin}{arg\,min}
\DeclareMathOperator{\ran}{ran \,}

\definecolor{darkcandyapplered}{rgb}{0.64, 0.0, 0.0}

\title{Deep neural networks can stably solve high-dimensional, noisy, non-linear inverse problems}

\author{Andrés Felipe Lerma Pineda\thanks{University of Vienna,
    Faculty of Mathematics, Kolingasse 14-16,
    1090 Wien,
    e-mail: \texttt{andres.felipe.lerma.pineda@univie.ac.at}
  } \and Philipp Christian Petersen\thanks{University of Vienna,
    Faculty of Mathematics and Research Network Data Science @ Uni Vienna, Kolingasse 14-16,
    1090 Wien,
    e-mail: \texttt{philipp.petersen@univie.ac.at}
  }}

\begin{document}

\date{\ }

\maketitle

\begin{abstract}
   We study the problem of reconstructing solutions of inverse problems when only noisy measurements are available. We assume that the problem can be modeled with an infinite-dimensional forward operator that is not continuously invertible. Then, we restrict this forward operator to finite-dimensional spaces so that the inverse is Lipschitz continuous. For the inverse operator, we demonstrate that there exists a neural network which is a robust-to-noise approximation of the operator. In addition, we show that these neural networks can be learned from appropriately perturbed training data. We demonstrate the admissibility of this approach to a wide range of inverse problems of practical interest. Numerical examples are given that support the theoretical findings.
\end{abstract}

\section{Introduction}\label{sec:Introduction}

In recent years, there has been an increasing interest in the field of \textit{deep learning} \cite{Goodfellow-et-al-2016, lecun2015deep}. The central concept behind deep learning is that of artificial neural networks. The name is inspired by the concept of neural networks in biology \cite{MP43}, although the functioning of both artificial and biological neural networks possesses profound differences (see \cite{ongie2020deep}). An artificial neural network can be viewed as a function with a composite structure which is typically organized in multiple layers. Neural networks have successfully formed the basis of several computational algorithms capable of performing different tasks ranging from image classification \cite{Krizhevsky2012Imagenet, hu2018squeeze, he2016deep} and voice recognition \cite{amodei2016deep} to natural language processing \cite{vaswani2017attention, bahdanau2014neural}. Nowadays, several researchers are focusing on tailoring and implementing neural-network-based techniques for problems traditionally in the area of applied mathematics, such as partial differential equations \cite{sirignano2018dgm, raissi2019physics, han2018solving}, dynamical systems \cite{brunton2022data, li2022deep}, or inverse problems \cite{arridge2019solving}. 
One theoretical reason for the success of deep learning is the fact that neural networks can efficiently approximate different functions belonging to various function spaces to high accuracies, see \cite{guhring2020expressivity, berner2021modern} for surveys. 


The main subject of this paper is the solution of infinite-dimensional, nonlinear, noisy inverse problems. Inverse problems are ubiquitous in natural sciences and applied mathematics \cite{kirsch2011introduction}. 
In simple terms, inverse problems are concerned with identifying the causes of a given phenomenon when it is (partially) known. However, due to restrictions in time, cost, or even impossibility to measure the causes directly, these causes are often impossible to obtain in practice. 
Many inverse problems of practical relevance are ill-posed in the sense of Hadamard \cite{engl1996regularization}. This implies that numerical reconstruction when only noisy data is available may lead to solutions differing substantially from the solution of the noiseless problem. As a remedy, various methods to solve inverse problems from noisy data have been designed, such as the well-studied \textit{regularization methods} \cite{kirsch2011introduction} and, more recently, data-driven methods \cite{arridge2019solving}. 

Inverse problems present themselves typically in the form of an equation $F(x)=y$ where $F$ is an operator between two suitable sets. The operator $F$ is called the \textit{forward operator} and is nonlinear in many cases of interest. The task of solving the inverse problem is to reconstruct $x$ given $F$ and a potentially inaccurate approximation of $y$. As discussed before, a continuous inverse of $F$ is rarely available and, therefore, noisy data can pose a significant challenge. One way to proceed is to determine when $F^{-1}$ restricted to a suitable domain can be well approximated by \emph{robust neural networks, i.e., neural networks implementing Lipschitz functions, which dampen the noise}. This is the approach taken in this work.

It is important to note that there is strong theoretical evidence for an intrinsic instability of neural-network-based reconstruction of inverse problems \cite{gottschling2020troublesome, colbrook2022difficulty, boche2022limitations}. 
Indeed, even for linear inverse problems it appears to be impossible to stably compute high-accuracy solutions using deep learning. In this work, we advocate for a different point of view. 
While deep neural networks are often found to be incapable of producing high-accuracy solutions even if just very small noise is present, we will demonstrate that \emph{in a high-noise-regime neural networks can produce very robust solutions compared to the noise level}. In fact, the noise is dampened if the trade-off between the intrinsic dimension of the problem and the dimension of the sampling space is favorable. 

We describe our contribution in detail in the following subsection.

\subsection{Our contribution}

We are concerned with establishing the existence and practical computability of noise dampening neural networks that approximate the inverse of a forward operator associated to an infinite-dimensional inverse problem. 
For the sake of simplicity, all inverse problems are modeled by an injective forward operator $F \colon \mathcal{X} \to \mathcal{Y}$, and therefore, each inverse problem possesses at most one solution. 
There is no a priori assumption on the continuity of the inverse of the forward operator $F$, but the spaces $\mathcal{X}$ and $\mathcal{Y}$ will be Banach spaces. 
In fact, we assume that the solution $x$ of the inverse problem $y= F(x)$ lies in a finite-dimensional space $W \subset \mathcal{X}$. Here we only have access to a noisy version of $y$, which we assume to be well approximated by a $y_\delta$ belonging to a $D$-dimensional space $Y_D \subset \mathcal{Y}$. 
The $Y_D$ correspond to approximations of elements from $\mathcal{Y}$ from $D$ measurements. 
Under a set of technical conditions on $\lGamma \subset W$, it can be derived that the restricted operator, $F|_{\lGamma}$, has a Lipschitz continuous inverse $F|_{\lGamma}^{-1}$ which is defined on $\mathcal{M} \coloneqq F|_{\lGamma}(\lGamma)$. 

The main contribution of this work is given by Theorem \ref{thm:LipschitzFunctionOnManifoldBetterEstimate2} and shows that a robust-to-noise neural network that approximates $F|_{\lGamma}^{-
1}$ can be constructed. 
A simplified version that contains the fundamental underlying ideas of the statement is the following:

\begin{theorem}\label{thm:intuitiveThm}

Let $q,d, D \in \mathbb{N}$ with $d < D$. Then, for every Lipschitz continuous function $f: \mani \subset \mathbb{R}^D \to \mathbb{R}^q$ such that $\mani$ is a $d-$dimensional submanifold satisfying some technical conditions, and for $\epsilon>0$, there exist neural network weights $\Phi^{f, \epsilon}_{\mani}$ with $D-$dimensional input and $\mathcal{O}(\epsilon^{-d})$ parameters such that 
$$
    \left\| f - \mathrm{R}\left(\Phi^{f, \epsilon}_{\mani} \right)\right\|_{L^\infty(\mani; \R^q)} \leq \epsilon,
$$
where $\mathrm{R}(\Phi^{f, \epsilon}_{\mani})$ denotes the so-called \textit{realization} of $\Phi^{f, \epsilon}_{\mani}$ (see Definition \ref{def:NeuralNetworks}), which is the function associated to the weights of the network. Furthermore, the resulting neural network is robust to normal-distributed noise, i.e., for all Gaussian random vectors $\eta \in \mathbb{R}^D$ with small enough variance, the following estimate holds 
\begin{align}\label{eq:TheProbabilisticStatement2}
    \mathbb{E}_\eta\left(\sup_{x \in \mani}\left|\mathrm{R}\left(\Phi^{f, \epsilon}_{\mani}\right)(x+\eta) - \mathrm{R}\left(\Phi^{f, \epsilon}_{\mani}\right)(x)\right|^2\right)  = \mathcal{O}\left(\mathbb{E}(|\eta|^2) \frac{d}{D} \right).
\end{align}
The implicit constant in the estimate depends on the Lipschitz constant of $f$ (Theorem \ref{thm:LipschitzFunctionOnManifoldBetterEstimate2} gives a more precise discussion of the implicit constants).
\end{theorem}
Equation \eqref{eq:TheProbabilisticStatement} illustrates what we mean by the noise-dampening effect of the neural network. For a typical 1-Lipschitz regular function $f: \R^D \to \R$, we would expect that for a $D$-dimensional Gaussian random vector $\eta$ it holds for an $x \in \R^D$ that $    \mathbb{E}\left(|f(x + \eta) - f(x)|^2\right) \leq \mathbb{E}\left(|\eta|^2\right).
$
Equation \eqref{eq:TheProbabilisticStatement} on the other hand, includes the additional dampening factor $d/D$. 

Notice that at first glance, it is not evident how Theorem \ref{thm:LipschitzFunctionOnManifoldBetterEstimate2} can be applied to solve inverse problems. Indeed, for most interesting inverse problems (as the ones discussed in this paper), the inverse of the forward operator is discontinuous and therefore not Lipschitz. Besides, the domain and codomain of the forward operator can be infinite-dimensional Banach spaces.
To apply Theorem \ref{thm:LipschitzFunctionOnManifoldBetterEstimate2}, we will approximate $F^{-1}$ by a Lipschitz continuous operator on finite-dimensional spaces. 
The existence of such an approximation, which corresponds to reconstruction from finitely many measurements, was established in \cite[Corollary 1]{alberti2019inverse} and the details of that result will be briefly discussed in Subsection \ref{subs:soluinvp}.
Moreover, Theorem \ref{thm:LipschitzFunctionOnManifoldBetterEstimate2} establishes the existence of a robust-to-noise neural network. Thus, after taking sufficiently many measurements and thereby transforming an infinite-dimensional problem into a stable finite-dimensional problem, Theorem \ref{thm:LipschitzFunctionOnManifoldBetterEstimate2} can be applied to obtain robust neural network approximations. 

The second main insight of the manuscript is that a robust-to-noise neural network of which we know its existence can be found with an appropriate training strategy. We will demonstrate in Theorem \ref{thm:noiseStabilityAfterTraining} that by training on randomly perturbed data, which can easily be generated in large quantities if only the forward operator of the associated inverse problem is known, we can find a robust neural network which has the same stability properties as the one of Theorem \ref{thm:intuitiveThm}/Theorem \ref{thm:LipschitzFunctionOnManifoldBetterEstimate2}.

The third contribution of this work is that we identify a comprehensive list of non-linear infinite-dimensional inverse problems to which the framework described above is applicable. The admissible problems are the transmissivity coefficient identification problem, the coefficient identification problem of the Euler-Bernoulli equation, an inverse problem associated to the Volterra-Hammerstein integral equation, and the linear gravimetric problem. 
We will describe these problems in detail in Section \ref{sec:admissibleInverseProblems}.

Finally, it is certainly not clear if the robust neural networks which are theoretically guaranteed to exist can be found in practice using gradient-based training. Thus, we add a comprehensive numerical study that clearly demonstrates that in many problems neural networks with the same stability properties as the theoretically established are found. This study is presented in Section \ref{sec:NumExps}.

We end this subsection by briefly discussing the ideas behind the proof of Theorem \ref{thm:intuitiveThm}/Theorem \ref{thm:LipschitzFunctionOnManifoldBetterEstimate2}. This result can be divided into two parts. The first one guarantees the existence of a relatively small neural network that approximates every Lipschitz continuous function defined on a low-dimensional manifold of a Euclidean space with arbitrary precision $\epsilon>0$. 
Although similar results are already well-known in the literature, the way in which the neural network is constructed differs from other approaches and our construction paves the way for the second statement of Theorem \ref{thm:intuitiveThm}/Theorem \ref{thm:LipschitzFunctionOnManifoldBetterEstimate2}. 
We construct the neural network not on the manifold $\mani$ but on $\mani+B_{\delta}(0)$, where for a $\delta>0$, $B_{\delta}(0)$ denotes the Euclidean ball of radius $\delta$ centered at the origin. Here $\delta$ depends on the manifold and will later determine the maximum noise level up to which the neural network is robust.
By making use of a suitable partition of unity $(\phi_i)_{i=1}^{M}$, for $M \in \N$, constructed from linear finite element functions, we can express $f$ as the following sum:
\begin{align}\label{eq:repOfManifoldFunction}
     f(x) =  \sum_{i = 1}^{M} \phi_{i}(x) f_i\left(\mathrm{P}_{y_i,\mani}(x)\right), \text{ for } x \in \mani 
\end{align}
where $f_i$, $i=1,\dots, M$ with $\supp(f_i) \subset B_\delta(y_i)$ for $y_i \in \mani$ are Lipschitz continuous functions and $\mathrm{P}_{y_i,\mani}$ is the projection to the tangent space of $\mani$ at $y_i$. To approximate \eqref{eq:repOfManifoldFunction} by a neural network, we notice, that a) we can exactly represent the partition of unity made from finite elements, thanks to a result of \cite{he2018relu}, b) multiplication of functions is efficiently emulated due to a result of \cite{PetV2018OptApproxReLU}, c) low-dimensional Lipschitz functions are again efficiently representable due to \cite{he2018relu}, and d) the projection operators are linear and hence correspond to neural networks without hidden layers. 
The construction of the overall neural network from the smaller building blocks as well as an estimate of the required number of neurons and layers thereof can be derived using a formal ReLU calculus, that will be recalled in Section \ref{sec:ReLUCalc}. 
To derive the robustness statement \eqref{eq:TheProbabilisticStatement2}, we first prove that, if $h_\epsilon$ denotes the neural network approximating $f$ up to an error of $\epsilon$ on $\mani+B_{\delta}(0)$ and $\eta \in \R^D$ is a Gaussian random variable, then, if $|\eta| \leq \delta$, for $x \in \mani$ 
$$
    |h_\epsilon(x+\eta) - h_\epsilon(x)|^2 \leq 2|f(x + \eta) - f(x)|^2 +  8\epsilon^2 \lesssim \max_{i \in [M]} |\mathrm{P}_{y_i, \mani} (\eta)|^2 + \epsilon^2
$$ 
since the $(f_i)_{i=1}^{M}$ are Lipschitz continuous. One can show, for all $i = 1, \dots, M$, that $ |\mathrm{P}_{y_i, \mani} (\eta)|^2/\sigma^2$ where $\sigma^2 = \mathbb{E}(\vert \eta \vert^2)$ is a $\chi_d^2$-distributed random variable and, therefore, by the maximal inequality for $\chi_d^2$ random variables \cite{boucheron2013concentration} the anticipated estimate follows. The result is completed by demonstrating that  $|\eta| > \delta$ is an event with very low probability.

\subsection{Related work}\label{subsec:RelatedW}

\subsubsection{Transformation of infinite dimensional inverse problems to finite dimensions} \label{subs:soluinvp}

In \cite{alberti2019inverse}, a strategy was introduced to transform infinite-dimensional discontinuous inverse problems into Lipschitz regular finite-dimensional problems. This approach forms the basis of our results and will be briefly discussed below.
Assuming again an inverse problem in the form of $F: \mathcal{X} \to \mathcal{Y}$, we choose a finite-dimensional space $W_0 \subset \mathcal{X}$ and restrict the forward operator to this space. 
However, restricting the domain only, does not guarantee that the image of the operator lies in a finite-dimensional linear space. To overcome this problem, a family of finite-rank bounded operators $Q_N: \mathcal{Y} \to \mathcal{Y}$, $N \in \N$, approximating the identity $I_{\mathcal{Y}}$ is introduced in \cite{alberti2019inverse} so that the function between finite-dimensional spaces $Q_N \circ F|_{W_0}: W_0 \to \mathrm{Im}(Q_N)$ can be interpreted as a continuous approximation of $F$. 
This approach fits into what can be expected in real-life problems. In practice, $Q_N$ corresponds to a sampling procedure. The following theorem was proved in \cite{alberti2019inverse}.

\begin{theorem}[{\cite[Corollary 1]{alberti2019inverse}}].\label{thm:Alberti}
Let $\mathcal{X}$, $\mathcal{Y}$ be Banach spaces, $A \subset \mathcal{X}$ be an open subset, $W_0 \subset \mathcal{X}$ be a finite-dimensional subspace and $K \subset W_0 \cap A$ be a compact and convex subset. If $F \in C^1(A, \mathcal{Y})$ is such that $F|_{W_0 \cap A}$ is injective and $F'(x)|_{W_0}$ is injective for all $x \in W_0 \cap A$, then there is a constant $C>0$ and $D \in \N$ such that
\begin{equation}
     \Vert x_1 - x_2 \Vert_\mathcal{X} \leq C \Vert Q_D(F(x_1)) - Q_D(F(x_2)) \Vert_\mathcal{Y},
 \end{equation}
 for all $x_1, x_2 \in K$.
\end{theorem}
Theorem \ref{thm:Alberti} states that the inverse of the function $Q_D \circ F|_K: K \to \mathrm{Im}(Q_D)$ is Lipschitz continuous. This observation will be significant for the solution of inverse problems by neural-network-based methods that will be proposed in this work. The constant $C>0$ can be estimated in terms of a lower bound of the Fr\'echet derivative and the moduli of continuity of $(F|_K)^{-1}$ and $F'$. In \cite{bacchelli2006lipschitz} and \cite{rondi2006remark} some ad-hoc techniques have been tailored to estimate such a constant $C$ for specific problems.

We would like to add that the method of \cite{alberti2019inverse}, was already used in the context of stabilizing deep-neural-network-based solutions of inverse problems in \cite{alberti2022inverseGenerative}.

\subsubsection{Approximation of functions on manifolds by neural networks}
It has been noticed in several works that for many reconstruction problems the data often lies in a low-dimensional manifold of a higher-dimensional space \cite{schmidt2019deep}. In that articles, all manifolds are subsets of a Euclidean space. Such a manifold is typically called the \textit{intrinsic space} and the higher-dimensional space is the \textit{ambient space}. Their dimensions are called \textit{intrinsic dimension} and \textit{ambient dimension}, respectively. Given data belonging to a manifold $\mani$, the problem of reconstructing its coordinate maps $\phi : U \subset \mani \to \mathbb{R}^d$, for $d \in \N$ by neural networks has been extensively studied (see \cite{bickel2007local, hein2005graphs, singer2006graph}). 
Besides, methods for the reconstruction of functions between two given manifolds by neural networks have been discussed \cite{mhaskar2010eignets, chui2018deep}, where smooth activation functions were considered. 
For the ReLU activation function, approximation of $C^k$-functions, $k \in \N$, on a manifold is analyzed in \cite{shaham2018provable, cloninger2021deep}.
Results for $\alpha$-H\"older continuous functions, $\alpha >0$, can be found in \cite{schmidt2019deep, kohler2019estimation, chen2019nonparametric, nakada2020adaptive}. Approximation of high-dimensional piecewise $\alpha$-H\"older continuous functions which include smooth dimension reducing feature maps was studied in \cite[Chapter 5]{PetV2018OptApproxReLU}. These works do not study the stability of the trained neural networks to noise in its input.

\subsubsection{Non-neural-networks-based approaches for ill-posed inverse problems}

Ill-posed inverse problems are typically approached by so-called \textit{regularization techniques}. There are at least four main categories in which regularization techniques can be classified. These categories were discussed in \cite{arridge2019solving} and will be briefly reviewed below.
The simplest of these approaches are the \textit{analytic inversion methods} which consist of determining a closed expression for $F^{-1}$ and attempt to stabilize it (see \cite{natterer2001mathematics, natterer2001mathematical}). Those methods provide well-posed solutions of linear inverse problems where the forward operator is compact and there is a spectral decomposition allowing a closed form of $F^{-1}$. The second category is the family of \textit{iterative methods} based on gradient-descent algorithms to minimize a functional $x \mapsto \Vert F(x)-y \Vert$ in a stable way (see \cite{engl1996regularization, kirsch2011introduction}).
The third category is that of the \textit{discretization methods} such as the \textit{Galerkin methods} where the solution 
is discretized to a finite-dimensional space. It is based on the idea that appropriate discretization stabilizes the inversion, e.g.,  \cite{natterer1977regularisierung, plato1990regularization}. Finally, the last category is composed of \textit{variational methods} that reconstruct the solution by minimizing a functional of the form
\begin{equation*}
    J_{\alpha}(x) = \Vert F(x)-y \Vert + S_{\alpha}(x)
\end{equation*}
where $S_{\alpha}: \mathcal{X} \to \mathbb{R}$ is chosen to enforce a priori known properties of the solution and $x \in \mathcal{X}$ (see \cite{kaltenbacher2008iterative}). 

\subsubsection{Deep-learning-based approaches to inverse problems}
Various deep-learning-based approaches have recently been applied to solve inverse problems. In \cite{ongie2020deep}, an extensive survey on data-driven solutions for inverse problems in imaging reviews some of the most prominent approaches. Although the main focus of the survey is on linear operators, the discussed methods can be implemented for non-linear operators as well. Data-driven techniques tackling the solution of an inverse problem are divided into four categories according to the knowledge about the forward operator $F$: in case it is known from the beginning \cite{gregor2010learning, sun2016deep}, during training \cite{dabov2007image, boyd2011distributed, venkatakrishnan2013plug}, partially known \cite{armanious2019unsupervised, zhu2017unpaired}, or never known \cite{zhu2018image}. 
As illustration of the applicability of these techniques, \cite{ongie2020deep} implements such neural-network-based techniques for the solution of problems in imaging, stressing how this novel approach has proved to be successful in several applications ranging from deblurring, super-resolution, and image inpainting, to tomographic imaging,  magnetic resonance imaging, X-ray computed tomography, and radar imaging. In addition, \cite{arridge2019solving} presents an extensive survey of data-driven methods used to regularize ill-posed inverse problems.

It has been shown that neural-network-based methods for the solution of inverse problems can be unstable (see \cite{antun2019instabilities, huang2018some}). In \cite{gottschling2020troublesome}, a comprehensive analysis explaining why these methods are intrinsically unstable is presented and the difficulties in finding remedies for these instabilities are also discussed. According to the authors, instabilities are not rare events and can easily destabilize trained neural networks. This analysis is performed for linear inverse problems, particularly for inverse problems in imaging. Our paper takes a different approach since we show that in an appropriate noise regime, robust neural networks exist and can be found by appropriate training. This phenomenon has also been observed in practice previously in \cite{genzel2022solving}, which is one of the main motivations for this work.

\subsection{Outline}
 In Section \ref{sec:admissibleInverseProblems}, we will introduce a list of inverse problems to which the general theory of \cite{alberti2019inverse} is applicable and which gives rise to inverse operators defined on smooth relatively compact manifolds. In Section \ref{sec:LipschitzFunctionOnManifoldBetterEstimate}, we show how the resulting functions representing the inverse operators can be approximated by neural networks in a way that is robust to noise. Section \ref{sec:statisticalLearning} discusses the design of an appropriate training set to learn robust neural networks. In Section \ref{sec:NumExps}, we collect some numerical examples to show under which conditions the noise-robust approximations are found by practical algorithms.

\section{Admissible inverse problems}\label{sec:admissibleInverseProblems}

Below, we will present a list of inverse problems that can be discretized in the way necessary for our theory to work. All these inverse problems are modeled by an operator $F$ between two Banach spaces $\mathcal{X}$ and $\mathcal{Y}$. We solve the inverse problem if the solution belongs to a manifold $\lGamma$ which is a subset of a compact and convex subset $K$ of a finite-dimensional subspace $W_0 \subset \mathcal{X}$. 
The first step is to choose a sequence of finite-rank operators $\widehat{Q}_{N} : \mathcal{Y} \to Y_N$ where $Y_N \subset \mathcal{Y}$ is an $N$-dimensional subspace for all $N \in \mathbb{N}$ such that $(\widehat{Q}_{N})_{N \in \N}$ is a sequence of operators converging strongly to the identity operator on $\mathcal{Y}$, i.e., for all $x \in \mathcal{X}$ it holds that $\Vert \widehat{Q}_N x - x \Vert_{\mathcal{Y}} \to 0$ as $N \to \infty$.
We identify the finite-dimensional spaces $W_0$ and $Y_N$ with $\mathbb{R}^d$ and $\mathbb{R}^N$ via affine isomorphims $P : \mathbb{R}^d \to W_0$ and $T_N \colon Y_N \to \mathbb{R}^N$. Then, the inverse problem of finding $x \in \lGamma$ given $y \in \ran \widehat{Q}_N \circ F$ modeled by the equation
$$
    \widehat{Q}_N \circ F(x) = y
$$
is equivalent to the inverse problem of finding $x' \in P^{-1}(\lGamma)$ given $y' \in \mathbb{R}^N$ such that
$$
    T_N \circ \widehat{Q}_N \circ F \circ P ( x') = y'
$$
where $x' = P^{-1}(x)$ and $y'= T_N(y)$. 

For each inverse problem that will be introduced below, we show that there is a $D \in \mathbb{N}$ such that the following discretized version of $F$,
\begin{align*}
    \widetilde{F} &\colon P^{-1}(\lGamma) \subset \mathbb{R}^d \to \mathbb{R}^D,\\
    \widetilde{F} &\coloneqq T_D \circ \widehat{Q}_{D} \circ F |_{K} \circ P, 
\end{align*}
where $F|_K$ is the restriction of $F$ to $K$, satisfies
\begin{enumerate}[(i)]
    \item $\ran \widetilde{F}$ is an $(M,\delta)$-covered submanifold $\mathcal{M}$. The notion of $(M, \delta)-$covered submanifold will be introduced in Definition \ref{def:MDelta},
    \item $\widetilde{F}$ is invertible on $\mathcal{M}$,
    \item $\widetilde{F}^{-1} \colon \mathcal{M} \to P^{-1}(\lGamma)$ is Lipschitz continuous.
\end{enumerate}


In the sequel, for the sake of clarity, we will always use $\lGamma$ to denote a manifold contained in the domain of a forward operator and $\mathcal{M}$ to denote a submanifold of the image. Since Theorems \ref{thm:LipschitzFunctionOnManifoldBetterEstimate2} will be applied to approximate the inverse of a forward operator, the function there is assumed to be defined on a manifold denoted by $\mathcal{M}$.

To show the conditions (i), (ii), and (iii) for the inverse problems introduced below, we derive closed expressions for the forward operators on certain domains $H \subset \mathcal{X}$. 
Then, we demonstrate that the respective forward operator $F$ is injective on $H$ and Fr\'echet differentiable at each $a \in H$. Next, we use Theorem \ref{thm:Alberti} to show that for every compact and convex set $K \subset W_0 \cap H$ where $W_0$ is a  $d$-dimensional subspace of $\mathcal{X}$, there is a $D \in \mathbb{N}$ such that $(\widehat{Q}_D \circ F|_K)^{-1}$ is Lipschitz continuous.

Thereafter, we identify the finite-dimensional spaces $W_0$ and $Y_D$ with the finite-dimensional spaces $\mathbb{R}^d$ and $\mathbb{R}^D$ through isomorphisms $P: \mathbb{R}^d \to W_0$ and $T_D: Y_D \to \mathbb{R}^D$. For the inverse problems presented in this paper, $P$ is usually an affine isomorphism and $T_D \circ \widehat{Q}_D$ is a point sampling operator. Hence, due to the composition of invertible functions we see that $\widetilde{F} $ defined as above is invertible and bi-Lipschitz, which implies that conditions (ii) and (iii) hold. We will see in all cases below, that the discretized operator $\widetilde{F}$ is $C^\infty$ and bi-Lipschitz. Hence it follows that $\widetilde{F}$ is a diffeomorphism which implies that $\mathcal{M}$ is an $(M,\delta)$-covered submanifold.

We discuss the transmissivity coefficient identification in Subsection \ref{sec:transmissivityIdentification}, as well as the coefficient identification problem of the Euler-Bernoulli equation in Section \ref{sec:EulerBernoulli}. Thereafter, we discuss inverse problems associated to the  Volterra-Hammerstein integral equation and the gravimetric problem in Sections \ref{sec:Volterra} and \ref{sec:gravimetric}, respectively.

\subsection{Identification of the transmissivity coefficient} \label{sec:transmissivityIdentification}

In this section, we introduce a problem arising from the study of aquifers (see \cite{niwas1985aquifer}) in groundwater hydraulics. An aquifer is a natural underground formation capable of storing and transmitting underground water. Mathematically speaking, an aquifer is a three dimensional set $V \subset \mathbb{R}^3$ with an associated potential energy per unit weight: $u \in C^2(V)$.
We can consider an aquifer as a one-dimensional object if its width and depth are negligible compared to its length or if its energy varies only with respect to one of its dimensions. In that case, the aquifer can be thought of as being the interval $I\coloneqq [0,1]$ where $0$ and $1$ are its endpoints and $x$ is referred as to the position or distance with respect to the first endpoint. Under certain circumstances, the energy transported by the water, $u \in C^2([0,1])$, is modeled by the equation 
\begin{align*}
\dfrac{d}{dx} \left( a(x) u'(x) \right)= f(x) \text{ for all } x\in [0,1],
\end{align*}
where the coefficient $a \in C^1([0,1])$ is referred to as the \emph{transmissivity coefficient} which measures the ability of the aquifer to allow the passing of groundwater and $f \in C([0,1])$ measures the change of mass over time at a given point \cite{yeh1986review}.
If $a \in C([0,1])$, then we assume the product of $a$ and $u'$ to be continuously differentiable. Solving the direct problem is to find $u$ given $f$ and $a$ under certain boundary or initial conditions. The inverse problem is then to find the coefficient $a$ when $u$ and $f$ are known. Note that if $u' > 0 $ and the value $c_0 = a(0)u'(0)$ with $c_0 \in \mathbb{R}$ is also known, the coefficient $a(x)$ at the point $x\in [0,1]$ can be directly reconstructed by the formula
\begin{align*}
    a(x) = \dfrac{\int_0^x f(t)dt + c_0}{u'(x)}.
\end{align*}
Nevertheless, this formula may be useless if noise is in the data 
or only point samples of the functions $f$, $u$ are known. 
If in addition to the condition $c_0 = a(0)u'(0) \in  \mathbb{R}$ mentioned above, a second condition $c_1 = u(0) \in \mathbb{R}$ is given, we express $u$ as a function of $a$ through the forward operator $F: H_{\lambda} \subset C([0,1]) \to C([0,1])$
\begin{align} \label{eq:transmissionForwardProblem}
u = F(a) \coloneqq \left( x \mapsto \int_0^x \dfrac{\int_0^z f(s)ds +c_0}{a(z)}dz  + c_1\right),
\end{align}
where $H_{\lambda} \coloneqq \{ a \in C([0,1]): a(x) > \lambda >0 \text{ for all } x \in [0,1]\}$ for fixed $\lambda >0$. 
Notice that $H_{\lambda}$ is an open set and if $f > 0$ and $c_0 = c_1=0$, then $F$ is an injective operator on $H_{\lambda}$. Moreover, we prove that this is a Fr\'echet-differentiable operator where its Fr\'echet derivative at every $a \in H_\lambda$ is computed below, via the Gateaux derivative at $a$ evaluated at $x\in [0,1]$ along the direction $\delta a \in C([0,1])$ denoted by $D_{\delta a} F(a)$.

\begin{proposition}\label{prop:direct}
Let $\lambda>0, H_{\lambda} \coloneqq \{ a \in C([0,1]): a(x) > \lambda >0 \text{ for all } x \in [0,1]\}$ and $F: H_{\lambda} \to C([0,1]) $ be the operator defined by
\begin{align}\label{eq:tras2}
F(a) \coloneqq \left( x \mapsto \int_0^x \dfrac{\int_0^z f(s)ds }{a(z)}dz  \right),
\end{align}
for all $a \in H_{\lambda}$ and $x \in [0,1]$. Then, $F$ is Gateaux-differentiable.
\end{proposition}

\begin{proof}
Due to the definition of the Gateaux derivative at $a \in H_{\lambda}$ in the direction $\delta a \in C([0,1]) $, we have for $x \in [0,1]$ that

\begin{align*}
D_{\delta a} F(a)(x) & =   \lim _{t \to 0} \dfrac{F(a+t \delta a)(x) - F(a)(x)}{t} \\
& =  \lim _{t \to 0} t^{-1} \cdot  \left( \displaystyle \int_0^x \dfrac{\int_0^z f(s)ds}{a(z) + t \delta a(z)}dz   -  \dfrac{\int_0^z f(s)ds}{a(z) }dz \right)\\
& =  \lim_{t \to 0} t^{-1} \cdot \left(\displaystyle \int_0^x \int_0^z f(s) ds \left ( \dfrac{1}{a(z)+ t \delta a(z)} - \dfrac{1}{a(z)} \right )dz\right)\\
& =  \lim_{t \to 0}   t^{-1} \cdot \left(\displaystyle t \int_0^x \int_0^z f(s) ds \left ( \dfrac{-\delta a(z)}{a(z)(a(z)+ t \delta a(z))} \right )dz\right)\\
& =  \lim_{t \to 0} \int_0^x \dfrac{ -\delta a(z) \int_0^z f(s)ds}{a(z)(a(z)+ t \delta a(z))}dz\\
& =   -\int_0^x \dfrac{ \delta a(z) \int_0^z f(s)}{a(z)^2}dz,
\end{align*}
where the last equality is an application of the dominated convergence theorem. 

\end{proof}

\begin{proposition}
Under the assumptions of Proposition \ref{prop:direct}, the operator $F$ defined by \eqref{eq:tras2} is Fr\'echet-differentiable.
\end{proposition}
\begin{proof}
Now, we prove that the Gateaux derivative at $a$ is indeed the Fr\'echet derivative of the operator at $a \in H_{\lambda}$. By the properties of the absolute value

\begin{align*}
& \left \vert \displaystyle \int_0^x \dfrac{\int_0^z f(s)ds}{a(z) + \delta a(z)}dz  - \left( \int_0^x \dfrac{\int_0^z f(s)ds}{a(z) }dz \right ) - \left (-\int_0^x \dfrac{ \delta a(z) \int_0^z f(s)ds}{a(z)^2}dz  \right) \right \vert \\
& = \left \vert \displaystyle \int_0^x \int_0^z f(s) ds \left ( \dfrac{-\delta a(z)}{a(z)(a(z)+ \delta a(z))} \right) dz + \int_0^x \dfrac{ \delta a(z) \int_0^z f(s)ds}{a(z)^2} dz \right \vert\\
&= \left \vert \displaystyle \int_0^x \left ( \int_0^z f(s) ds \right) \left ( \dfrac{(\delta a(z))^2}{a(z)^2(a(z)+ \delta a(z))}\right)dz \right \vert\\
& \leq   \Vert \delta a \Vert_{\infty}^2  \sup_{x \in [0,1]}\left \vert \displaystyle \int_0^x \left (  \int_0^z f(s) ds \right) \left ( \dfrac{dz}{a(z)^2(a(z)+ \delta a(z))}\right) \right \vert,
\end{align*}
and therefore
\begin{align*}
&\lim_{\delta a \to 0} \dfrac{\Vert F(a+\delta a) - F(a) - D_{\delta a} F(a) \Vert_{\infty}}{\Vert \delta a \Vert_{\infty}}\\
& \leq \lim_{\delta a \to 0}  \Vert \delta a \Vert_{\infty}^2  \sup_{x \in [0,1]}\left \vert \displaystyle \int_0^x \left (  \int_0^z f(s) ds \right) \left ( \dfrac{dz}{a(z)^2(a(z)+ \delta a(z))}\right) \right \vert / \Vert \delta a \Vert_{\infty} = 0,
\end{align*}
since for $\|\delta a \|_\infty$ small enough, we have that 
$$
\sup_{x \in [0,1]}\left \vert \displaystyle \int_0^x \left (  \int_0^z f(s) ds \right) \left ( \dfrac{dz}{a(z)^2(a(z)+ \delta a(z))}\right) \right \vert 
$$
is bounded. Therefore, the Fr\'echet derivative for the operator $F$ at $a$ denoted by $F_a': C([0,1]) \to C([0,1])$ is given by
$$
F_a'(\delta a) = \left ( x \mapsto -\int_0^x \dfrac{ \delta a(z) \int_0^z f(s)ds}{a(z)^2}dz \right ) \quad 
\text{for} \quad \delta a \in C([0,1]).
$$
\end{proof}

\begin{proposition}
Under the assumptions of Proposition \ref{prop:direct}, let us consider the operator $F$ defined by \eqref{eq:tras2}. Let $F_a'$ denote the Fr\'echet derivative of $F$ at $a \in H_{\lambda}$.
Then, $F_a'$is injective for all $a \in H_{\lambda}$.
\end{proposition}

\begin{proof}

 As $F'_a$ is a linear operator, it is enough to prove that if $F_a'(\delta a) = 0$ for $\delta a \in C([0,1])$,  then $\delta a =0$. If for all $x \in [0,1]$
 
\begin{align*}
F_a'(\delta a)(x) & = -\int_0^{x} \dfrac{ \delta a(z) \int_0^z f(s)ds}{a(z)^2}dz =0,\\
\end{align*}
then, due to the continuity of the integrand and by the fundamental theorem of calculus, we have that
\begin{align*}
\dfrac{ \delta a(z) \int_0^z f(s)ds}{a(z)^2}=0
\end{align*}
which shows that $\delta a(z) = 0$ for all $z \in [0,1]$.
 
 \end{proof}


We have observed above that the operator $F$ is injective and Fr\'echet differentiable and $F_a'$ is injective for all $a \in H_{\lambda}$. 

Next, we introduce the operator $P_{W, \lambda}: [a,b]^d \subset \mathbb{R}^d \to W+ \lambda$ given by $P_{W, \lambda}(\alpha) = \lambda + \sum_{k=1}^d \alpha_k \phi_k$,
where $(\phi_k)_{k=1}^d$ is a basis for a finite dimensional space $W$ where $\phi_i>0$ for all $i \in [d]$ and $\alpha \in [a,b]^d$ for $0<a<b$ is an isomorphism onto its range. Let $(Q_N)_{N\in \N}$ be a sequence of operators converging strongly to the identity operator as in the beginning of Section \ref{sec:admissibleInverseProblems}.

Let  $W_0 \coloneqq W \oplus \lambda$, then, according to Theorem \ref{thm:Alberti} there is a $D \in \mathbb{N}$ and $C>0$ such that 
$$
    \Vert a_1 - a_2 \Vert_{\infty} \leq C \Vert Q_D(F(a_1)) - Q_D(F(a_2)) \Vert_{\infty},
$$
for all $a_1, a_2$ in an arbitrary compact and convex $K \subset W_0 \cap H_{\lambda}$. Notice that this result implies that, for an isomorphism $T_D$ such that $T_D \circ Q_D$ is a point sampling operator, we have that $T_D \circ Q_D \circ F$ has a Lipschitz continuous inverse. 

We conclude that the operator $\widetilde{F} = T_D \circ Q_D \circ F \circ P_{W, \lambda}$ is an invertible and bi-Lipschitz function. Additionally, it is clear by construction that for every $x \in [0,1]$
\begin{align*}
    \R^d \ni \alpha \mapsto \int_0^x \dfrac{\int_0^z f(s)ds}{( \lambda + \sum_{k=1}^d \alpha_k \phi_k(z))}dz
\end{align*}
is smooth. This implies that $\widetilde{F}$ is smooth as well. Notice that $\widetilde{F}$ returns a $D$-dimensional vector which corresponds to the values of the functions at the sampling points. The details behind this approach can be found on \cite[Page 7]{alberti2019inverse}. We conclude with the inverse function theorem that $\mathcal{M} = \widetilde{F}((a,b)^d)$ is a relatively compact manifold.

\subsection{Identification of the coefficient in the Euler-Bernoulli equation}\label{sec:EulerBernoulli}

A beam is a long rigid element of a complex structure which is subject to a usually perpendicular external force causing the beam to bend. For simplicity, we think of a beam as the set of points $B=[0,1] \times \{ 0 \} $.
The deflection of the point $(x,0)$ is the measure of how much it is moved to a new position $(x', y') \in \mathbb{R}^2$. The model where deflections occur only in the $y$-coordinate under the action of a perpendicular force acting on $(x,0)$ is called the \textit{Euler-Bernoulli} model. Such deflections also depend on a material property called \textit{flexural rigidity} denoted by $a$. 
By simplicity, we identify the point $(x, 0)$ with $x\in [0,1]$, the $y$-coordinate of the deflection acting on $x$ as $u(x) = u(x,0)$, and the $y$-coordinate of the force acting on $x$ as $f(x)$. We call $u: [0,1] \to \mathbb{R}$ the \textit{deflection function} and $f:[0,1] \to \mathbb{R}$ the \textit{force}. If $u \in C^4([0,1])$, $a\in C([0,1])$ and $f\in C([0,1])$, then, under some additional conditions that are described in \cite{marinov2008inverse}, the deflection $u(x)$ of the beam at the point $x$ is modeled by the differential equation
\begin{align*}
\dfrac{d^2}{dx^2} \left ( a(x) u''(x) \right ) & =  f(x) \quad  \text{ for } x \in [0,1],
\end{align*}
which is called the \emph{Euler--Bernoulli} equation. There is a direct problem associated to this phenomenon consisting of identifying the deflection $u(x)$ at every point $x\in [0,1]$ when  $f$ and $a$ are given and the following boundary conditions hold
\begin{align*}
(a u'')'(0)& = c_3,\\
a(0) u''(0) & =  c_2,\\
u'(0) &= c_1,\\
u (0) & = c_0,
\end{align*}
where $c_0$, $c_1,$ $c_2$ $c_3 \in \mathbb{R}$. There are two inverse problems associated to this forward problem. First, the problem of reconstructing the coefficient $a$ leading to the so-called \textit{coefficient identification problem}. 
This problem is particularly relevant when studying structures eroded by the action of natural agents such as storms and earthquakes and if it is impossible to dismantle the whole structure. 
Second, we have the (inverse) problem of reconstruction of the source $f$. 
In this work, we focus only on the first of the two inverse problems.

We remark that for the coefficient identification problem as described above the solution can be exactly reconstructed by integrating the differential equation. Concretely, if $u''(x) \neq 0$ for all $x \in [0,1]$ an explicit formula for the reconstruction of the coefficient $a$ at the point $x \in [0,1]$ is given by
\begin{align}\label{eq:Euler}
a(x) = \dfrac{\int_0^x \int_0^t f(s)ds dt+c_3 x + c_2}{u''(x)}. 
\end{align}
However, when dealing with real-life problems, the exact value of $u$ is unknown and instead, all that is provided is an approximation with noise. 
For our method to be applicable, we assume that $a$ belongs to a finite-dimentional space. Straightforward calculations lead to the forward operator $F: H_\lambda \subset C([0,1])  \to C([0,1])$

\begin{align}\label{eq:EulerBernoulli}
F(a) \coloneqq \left ( x \mapsto \int_0^x \int_0^y \dfrac{\left( \int_0^s \int_0^w f(z)dzdw + c_3s + c_2 \right)}{a(s)}dsdy + c_1x+c_0 \right), 
\end{align}
where $H_\lambda \coloneqq \{a \in C([0,1]): a > \lambda >0\}$ for a fixed $\lambda$ is an open set. The uniqueness of $a$ for $u \in \mathrm{Im}(F)$ such that $u''>0$ and for initial conditions $c_0=c_1=c_2=c_3=0$ is guaranteed by the explicit formula \eqref{eq:Euler}, i.e., $F$ is injective. 
Similar calculations to the example of Subsection \ref{sec:transmissivityIdentification} can be carried out to obtain the derivative of $F$ as
\begin{align*}
    F'(a)(\delta a) = -\int_0^x \int_0^y \dfrac{\delta a(s) \left( \int_0^s \int_0^w f(z)dz dw  \right)}{a(s)^2}ds dy, \text{ for } \delta a \in C([0,1]).
\end{align*}
Moreover, by similar arguments as in the previous subsection, it can be shown that $F'(a)$ is injective for a fixed $a$ and continuous if $f>0$. 

We define again $P_{W, \lambda}: \mathbb{R}^d \to W$ given by $P_{W, \lambda}(\alpha) = \lambda + \sum_{k=1}^d \alpha_k \phi_k$,
where $(\phi_k)_{k=1}^d$ is a basis for a subspace $W$, $\phi_i>0$ for all $i \in [d]$, and $\alpha \in [a,b]^d$ for $0<a<b$. Let $(Q_N)_{N\in \N}$ be a sequence of operators converging strongly to the identity operator as in the beginning of Section \ref{sec:admissibleInverseProblems}. Let $W_0 \coloneqq W \oplus \lambda$. 
Then, Theorem \ref{thm:Alberti} can be applied again. This shows that there exists a $D \in \mathbb{N}$ and $C>0$ such that 
$$
    \Vert a_1 - a_2 \Vert_{\infty} \leq C \Vert Q_D(F(a_1)) - Q_D(F(a_2)) \Vert_{\infty},
$$
for all  $a_1, a_2$ in an arbitrary compact and convex $K \subset W_0 \cap H_{\lambda}$. We conclude that $Q_D \circ F|_K$ is bi-Lipschitz and injective. In addition, for an isomorphism $T_D$ such that $T_D \circ Q_D$ is a point sampling operator, we have that $T_D \circ Q_D \circ F$ has a Lipschitz continuous inverse. 

Hence, the operator $\widetilde{F} = T_D \circ Q_D \circ F \circ P_{W, \lambda}$ is an invertible, bi-Lipschitz, and smooth function. As previously discussed, $\mathcal{M} = \widetilde{F}((a,b)^d)$ is a relatively compact manifold.

\subsection{Solution of a non-linear Volterra-Hammerstein integral equation}\label{sec:Volterra}

Volterra-Hammerstein equations have been widely studied in the literature \cite{miller1975volterra, razzaghi2001solution, sepehrian2005solution}. Let $F: L^2([0,1]) \to L^2([0,1])$ be the following (forward) operator
\begin{align}
\label{eq:VolterraHammerstein}
    u \mapsto  F(u) \coloneqq \left ( t \mapsto \int_0^t (u(s))^2 ds \right ),
\end{align}
for $u \in L^2([0,1])$ and $t \in [0,1]$. 
The direct problem is to determine the function $F(u) \in L^\infty([0,1]) \subset L^2([0,1])$ for a given $u \in L^2([0,1])$. 
Let us notice that $F(u)$ always exists for $u \in L^2([0,1])$. Let us consider the set $H_{\lambda} = \{ u \in L^2([0,1]): u > \lambda >0\}$ for fixed $\lambda$. 
Notice that $F(H_{\lambda}) \subset C^1([0,1])\cap \{ v \in C^1([0,1]) : v'>0 \} \subset L^2([0,1])$. 
A solution to the inverse problem is found via the fundamental theorem of calculus. For every $v \in C^1([0,1])$, such that $v'>0$ the equation $F(u) = v$ has two continuous solutions $u_+ = \sqrt{v'}$ and $u_- = - \sqrt{ v'}$. As $u_- \notin H_{\lambda}$, we conclude that the operator $F|_{H_{\lambda}}$ is injective and this inverse problem has a unique solution. It is a well-known result that integral linear operators on infinite-dimensional spaces are compact operators and therefore, their inverses are not bounded. Hence, the inverse operator of \eqref{eq:VolterraHammerstein} is not continuous. 
To overcome the instability of the inverse problem, we use Corollary 1 in \cite{alberti2019inverse}. By similar calculations to the previous examples, the first Fr\'echet derivative at $u$ of the forward operator denoted by $F_u': L^2([0,1]) \to L^2([0,1])$ is the linear operator 
$$
   \delta u \mapsto F_u'(\delta u) = \left( t \mapsto  2 \int_0^t u(s) \delta u(s) ds \right ),
   $$
for all $\delta u \in L^2([0,1])$. 
Notice that under the assumption $u(s)>0$, we have that $F_u'(\delta u)=0$ if and only if $\delta u =0$. Indeed, we have that if $F_u'(\delta u) = F_u'(\delta u')$ for $\delta u, \delta u' \in L^2([0,1])$, then
$$
\int_0^x u(s)(\delta u(s) - \delta u'(s)) ds = 0,
$$
by the Lebesgue differentiation theorem we have that for almost every $t \in [0,1]$
\begin{align*}
    u(t)(\delta u(t) - \delta u'(t)) & = \lim_{h \to 0} h^{-1}\int_{t-h}^{t+h} u(s)(\delta u(s) - \delta u'(s))ds \\
    & = \lim_{h \to 0} h^{-1}\cdot \left (  \int_{0}^{t+h} u(s)(\delta u(s) - \delta u'(s))ds - \int_{0}^{t-h} u(s)(\delta u(s) - \delta u'(s))ds\right)\\
    & = 0,
\end{align*}
which implies $\delta u = \delta u'$ and therefore $F_u'$ is injective. 

Since the operator $F$ is Fr\'echet differentiable, injective, and $F_u'$ is injective for all $u \in C([0,1]) $, Theorem \ref{thm:Alberti} can be applied as in the previous examples. 
The discussion now follows that of Subsection \ref{sec:transmissivityIdentification}. 

We define again $P_{W, \lambda}: \mathbb{R}^d \to W+ \lambda$ given by $P_{W, \lambda}(\alpha) =  \sum_{k=1}^d \alpha_k \phi_k+\lambda$,
such that $\lambda>0$ and $(\phi_k)_{k=1}^d$ is a basis of a subspace $W$ where $\phi_k>0$ for $k \in [d]$, $\alpha \in [a,b]^d$ for $0<a<b$.
Let $W_0 = W \oplus \lambda$, then Theorem \ref{thm:Alberti} yields that for any sequence $(Q_N)_{N \in \mathbb{N}}$ converging strongly to the identity operator as $N \to \infty$, there is $D \in \mathbb{N}$ and $C>0$ such that 
$$
    \Vert u_1 - u_2 \Vert_{2} \leq C \Vert Q_D(F(u_1)) - Q_D(F(u_2)) \Vert_{2},
$$
for all $u_1, u_2$ in a compact and convex $K \subset W_0 \cap H_{\lambda}$. Following a similar discussion as in Section \ref{sec:transmissivityIdentification}, it follows that the operator $\widetilde{F} = T_D \circ Q_D \circ F|_K \circ P_{W}$ is an invertible, bi-Lipschitz, and smooth function between Euclidean spaces and $\mathcal{M} = \widetilde{F}((a,b)^d)$ is a relatively compact manifold.

\subsection{Inverse gravimetric problem}\label{sec:gravimetric}

An object with shape $\mathcal{E} \subset \mathbb{R}^2$ and mass density  $\rho: \mathcal{E} \to \mathbb{R}$ admits a gravitational potential $U_{\mathcal{E}, \rho}: \mathbb{R}^2 \setminus \mathcal{E} \to \mathbb{R}$ at the position $y \in \mathbb{R}^2$ given by
\begin{align}
    \label{eq:gravimetric}
    U_{\mathcal{E}, \rho}(y) \coloneqq C \int_{\mathcal{E}} \rho(x) \ln(|x-y|)dx,
\end{align}
where $C$ is called the \textit{gravitational constant} which is set to be 1 for simplicity (see \cite{kontak2018greedy}). The direct problem associated to this operator consists of determining the gravitational potential $U_{\mathcal{E}, \rho}$ when $\rho$ and $\mathcal{E}$ are known. 
Two inverse problems are linked to this equation. The first problem is called the \textit{linear inverse gravimetric problem}. 
Since it is physically impossible to measure the gravitational potential at every point, measurements are only taken at points on a surface $S$ away from $\bar{\mathcal{E}}$. 
We assume that the distance between an element $y \in S$ and $\mathcal{E}$ is $d(y, \mathcal{E})>\delta$ for a fixed $\delta>0$. For simplicity, we assume that $S \subset \mathbb{R}^2$ is a closed piecewise smooth surface which is the boundary of a solid body $B \subset \mathbb{R}^2$. The linear inverse gravimetric problem consist of finding only the mass density $\rho$ such that
$$
    U_{\mathcal{E}, \rho} |_S = g,
$$
when $U_{\mathcal{E}, \rho} |_S $, $\mathcal{E}$ and $g \in L^2(S)$ are known. The second inverse gravimetric problem is also known as \textit{the nonlinear inverse gravimetric problem} and is to  find a shape $\mathcal{E}$ such that $\bar{\mathcal{E}}\subset \mathrm{int}(B)$ and
$$
    U_{\mathcal{E}, \rho} |_S = g.
$$
This work focuses on the linear inverse problem.  We consider the operator
\begin{align*}
F \colon L^2(\mathcal{E}) &\to L^2(S), \\
    \rho &\mapsto  F(\rho) \coloneqq \left ( y \mapsto U_{\mathcal{E}, \rho}(y) \right).
\end{align*}
To demonstrate the smoothness of $F$ showing the continuity will suffice since $F$ is linear (in $\rho$). 
To prove the continuity of $F$, it will be shown that the operator is compact, which means that for every bounded family $\mathcal{F}$ of functions in $L^2(\mathcal{E})$, the family of functions $F(\mathcal{F})$ is relatively compact in $L^2(S)$. This will be established by invoking the theorem of Arzela-Ascoli (see \cite{munkres2000topology}):
every non-empty family of functions $\mathcal{G}$ where the elements are pointwise bounded and $\mathcal{G}$ is equicontinuous is relatively compact. 
First, the equicontinuity condition for $F(\mathcal{F})$ will be analyzed. Given $y, z \in S$ and $\rho \in \mathcal{F}$ it follows
\begin{equation*}
    |F(\rho)(y) - F(\rho)(z)| \leq  C \int_{ \mathcal{E}} \vert \rho(x) \vert \left \vert \ln{|y-x|} - \ln{|z-x|}  \right \vert dx
\end{equation*}
Since $f(t) = \ln(t)$ is Lipschitz continuous in every interval $[\delta, \infty)$ with constant $C' \leq 1/\delta$, for $t<s$ it holds that
\begin{eqnarray*}
\vert \ln{t} - \ln{s} \vert  \leq C' |t-s|. 
\end{eqnarray*}
Choosing $\delta>0$ such that $d(y, \mathcal{E})> \delta$ for all $y \in S$, and setting $t=|y-x|$ and $s=|z-x|$, we have by the inverse triangle inequality that
\begin{eqnarray*}
    \left \vert \ln{|y-x|} - \ln{|z-x|} \right \vert \leq C' \left | |y-x| - |z-x| \right | \leq C'|y-z|,
\end{eqnarray*}
and therefore
\begin{eqnarray*}
    |F(\rho)(y) - F(\rho)(z)| & \leq & C \int_{ \mathcal{E}} \vert \rho(x) \vert \left  \vert  \ln{|y-x|} - \ln{|z-x|}  \right \vert dx \\
                      & \leq & C'C \vert z-y \vert \int_{ \mathcal{E}} \vert \rho(x) \vert  dx \\
                      & \leq  & \hat{C} \Vert \rho \Vert_{L^1 ( \mathcal{E}  )} \vert z-y \vert \\
                      & \leq & \hat{C} C^{*} \Vert \rho\Vert_{L^2 ( \mathcal{E}  )} \vert z-y \vert\\
                      & \leq & \widetilde{C} \Vert \rho \Vert_{L^2 ( \mathcal{E}  )} \vert z-y \vert,
\end{eqnarray*}
where $\hat{C} \coloneqq C' C$, we used that $\Vert \rho\Vert_{L^1 ( \mathcal{E}  )} \leq C^{*} \Vert \rho\Vert_{L^2 ( \mathcal{E}  )}$ for a universal constant $C^{*}>0$ since $\mathcal{E} $ is bounded, and $\widetilde{C} \coloneqq \hat{C} C^{*} $. Since $\mathcal{F}$ was a bounded subset of $L^2(\mathcal{E})$, we conclude that  $F(\mathcal{F})$ is equicontinuous.

Now, the point-wise boundedness of $F(\mathcal{F})$ will be derived. For $y \in \mathcal{E}$, the function $f_y(x) = \vert \ln{\vert x-y \vert} \vert$ is continuous on the compact set $S$. Then, for $y \in \mathcal{E}$, the function $f_y(x)=|\ln(\vert y-x \vert)|$ is bounded with $ f_y(x) \leq M_y$ where $M_y>0$ is a constant that depends on $y$. Hence,
\begin{eqnarray*}
\vert F(\rho)(y) \vert & \leq & C \int_{ \mathcal{E}} \vert \rho(x) \vert  \vert \ln{|y-x|} \vert dx \\
                 & \leq & C  M_y \int_{ \mathcal{E}} \vert \rho(x) \vert dx \\
                 & = & C C^* M_y \Vert \rho \Vert_{L^1(\mathcal{E})}\\
                 & \leq & C M_y C^* \Vert \rho \Vert_{L^2(\mathcal{E})}\\
                 & \leq & \widetilde{C} M_y, 
\end{eqnarray*}
where $\widetilde{C} = C^*C$. Hence, $F(\mathcal{F})$ is pointwise bounded and therefore relatively compact. We conclude that the inverse problem associated to this operator is ill-posed.

Next, we introduce the operator $P_{W, \lambda}: (a,b)^d \subset \mathbb{R}^d \to W$ given by $P_{W, \lambda}(\alpha) = \sum_{k=1}^d \alpha_k \phi_k$,
where $(\phi_k)_{k=1}^d$ is a basis for a finite-dimensional space $W$ where $\phi_i>0$ for all $i \in [d]$ and $\alpha \in [a,b]^d$ for $0<a<b$ is an isomorphism onto its range. It follows that the conditions of Theorem \ref{thm:Alberti} are satisfied.
For any sequence $(Q_N)_{N \in \mathbb{N}}$ of finite-rank operators, where $Q_N: L^2(S) \to L^2(S)$ and $Q_N \to I$ strongly as $N \to \infty$, there is $D \in \mathbb{N}$ and $C>0$ such that 
$$
    \Vert \rho_1 - \rho_2 \Vert_{2} \leq C \Vert Q_D(F(\rho_1)) - Q_D(F(\rho_2)) \Vert_{2},
$$
for all $\rho_1, \rho_2$ in an arbitrary compact and convex $K$ subset of $W$. We can see that the operator $\widetilde{F} = T_D \circ Q_D \circ F|_K \circ P_{W}$ is an invertible, bi-Lipschitz, and smooth function between Euclidean spaces and $\mathcal{M} = \widetilde{F}((a,b)^d)$ is a relatively compact manifold.

\section{Approximation of Lipschitz continuous functions on smooth \\ manifolds by neural networks} \label{sec:LipschitzFunctionOnManifoldBetterEstimate}

Given $d, D \in \N$, we are aiming to approximate regular functions defined on smooth, $d$-dimensional submanifolds $\mani\subset \R^D$ by neural networks . Our focus here is to understand the extent to which the resulting neural networks are robust to noisy perturbations of the input. The main result is Theorem  \ref{thm:LipschitzFunctionOnManifoldBetterEstimate2}.

To be able to state and prove this result, we first need to introduce some notions associated to neural networks.

\begin{definition}[{\cite{PetV2018OptApproxReLU, FEMNNsPetersenSchwab}}]\label{def:NeuralNetworks}
Let $d, L\in \N$. 
A \emph{neural network (NN) with input dimension $d$ and $L$ layers} 
is a sequence of matrix-vector tuples 
\[
    \Phi = \big((A_1,b_1),  (A_2,b_2),  \dots, (A_L, b_L)\big), 
\]
where $N_0 \coloneqq d$ and $N_1, \dots, N_{L} \in \N$, and 
where $A_\ell \in \R^{N_\ell\times N_{\ell-1}}$ and $b_\ell \in \R^{N_\ell}$
for $\ell =1,...,L$.

For a NN $\Phi$ and an activation function $\varrho: \R \to \R$, 
we define the associated
\emph{realization of $\Phi$} as 
\[
 \mathrm{R}(\Phi): \R^d \to \R^{N_L} : x\mapsto x_L := \mathrm{R}(\Phi)(x),
\]
where the output $x_L \in \R^{N_L}$ results from 
\begin{equation}
    \begin{split}
        x_0 &\coloneqq x, \\
        x_{\ell} &\coloneqq \varrho(A_{\ell} \, x_{\ell-1} + b_\ell) \quad \text{ for } \ell = 1, \dots, L-1,\\
        x_L &\coloneqq A_{L} \, x_{L-1} + b_{L}.
    \end{split}
    \label{eq:NetworkScheme}
\end{equation}
Here $\varrho$ is understood to act component-wise on vector-valued inputs, 
i.e., for $y = (y^1, \dots, y^m) \in \R^m$,  $\varrho(y) := (\varrho(y^1), \dots, \varrho(y^m))$.
We call $N(\Phi) \coloneqq d + \sum_{j = 1}^L N_j$ the \emph{number of neurons of} 
$\Phi$, $L(\Phi)\coloneqq L$ the \emph{number of layers} or \emph{depth}, 
$W_j(\Phi)\coloneqq \| A_j\|_{0} + \| b_j \|_{0} $ the \emph{number of weights in the $j$-th layer}, and
$W(\Phi) \coloneqq \sum_{j=1}^L W_j(\Phi)$ the \emph{number of weights of $\Phi$}, 
also referred to as the \emph{size} of $\Phi$. 
The number of weights in the first layer is also denoted by $\sizefirst(\Phi)$, 
the number of weights in the last layer by $\sizelast(\Phi)$.
We refer to $N_L$ as the \emph{dimension of the output layer of $\Phi$}. Lastly, we refer to $(d, N_1, \dots, N_L)$ as the \emph{architecture of $\Phi$}.
\end{definition}

From now on, we will restrict ourselves to the most commonly used activation function $\varrho: \mathbb{R} \to \mathbb{R}$ given by $ \varrho(x) = \max\{0,x\}$ which is called \emph{Rectified Linear Unit (ReLU)}. We proceed by fixing a formal framework for the manipulation of NNs.
\subsection{ReLU calculus}\label{sec:ReLUCalc}

Our goal is to formally describe certain operations with NNs that mirror standard operations on functions such as composition or addition. In this regard, we recall three results of \cite{PetV2018OptApproxReLU} and one of \cite{EGJS2018} below. These results can also be understood as the definitions of the associated procedures. We start with \emph{concatenation of NNs}. For a better understanding of these operations, we refer to \cite[Section 2]{petersen2020neural}, where several useful illustrations can be found for intuition.

\begin{proposition}[NN concatenation \cite{PetV2018OptApproxReLU}]\label{prop:conc}
Let $L_1, L_2 \in \N$, and let 
$\Phi^1, \Phi^2$ 
be two NNs of respective depths $L_1$ and $L_2$ such that $N^1_0 = N^2_{L_2}\eqqcolon d$, i.e.,
the input layer of $\Phi^1$ has the same dimension as the output layer of $\Phi^2$. 

Then, there exists a NN $\Phi^1 \sconc \Phi^2$, called 
the \emph{sparse concatenation of $\Phi^1$ and $\Phi^2$}, 
such that $\Phi^1 \sconc \Phi^2$ has $L_1+L_2$ layers,   
$\mathrm{R}(\Phi^1 \sconc \Phi^2) = \mathrm{R}(\Phi^1) \circ \mathrm{R}(\Phi^2)$, 
\begin{align*}
\sizefirst(\Phi^1 \sconc \Phi^2) 
	\leq \begin{cases}
		2 \sizefirst(\Phi^2) & \text{ if }
		L_2 = 1,
		\\
		\sizefirst(\Phi^2) & \text{ else,}
	\end{cases}
	\qquad
\sizelast(\Phi^1 \sconc \Phi^2) 
	\leq \begin{cases}
		2 \sizelast(\Phi^1) & \text{ if }
		L_1 = 1,
		\\
		\sizelast(\Phi^1) & \text{ else,}
	\end{cases}
\end{align*}
and 
$$
W\left(\Phi^1 \sconc \Phi^2\right) \leq W\left(\Phi^1\right) + 
     \sizefirst \left(\Phi^1\right) + \sizelast \left(\Phi^2\right) + W\left(\Phi^2\right) 
  \leq  2W\left(\Phi^1\right)  + 2W\left(\Phi^2\right) .
$$ 
\end{proposition}
We introduce the parallelization of NNs next.
\begin{proposition}[NN parallelization \cite{PetV2018OptApproxReLU}]\label{prop:parall}
Let $L, d \in \N$ and let 
$\Phi^1, \Phi^2$ 
be two NNs with $L$ layers and with $d$-dimensional input each.
Then there exists a NN 
$\mathrm{P}(\Phi^1, \Phi^2)$ with $d$-dimensional input and $L$ layers, 
which we call the \emph{parallelization of $\Phi^1$ and $\Phi^2$}, 
such that 
\begin{equation}
\Realization\left(\mathrm{P}\left(\Phi^1,\Phi^2\right)\right) (x) 
= 
\left(\Realization\left(\Phi^1\right)(x), \Realization\left(\Phi^2\right)(x)\right), 
\text{ for all } x \in \R^d,
\label{eq:ParallelizationDoesTheRightThing}
\end{equation}
$W(\mathrm{P}(\Phi^1, \Phi^2)) = W(\Phi^1) + W(\Phi^2)$, 
$\sizefirst(\mathrm{P}(\Phi^1, \Phi^2)) = \sizefirst(\Phi^1) + \sizefirst(\Phi^2)$ 
and
$\sizelast(\mathrm{P}(\Phi^1, \Phi^2)) = \sizelast(\Phi^1) + \sizelast(\Phi^2)$.
\end{proposition}
Proposition \ref{prop:parall} requires two NNs to have the same number of layers to be put in parallel. However, there is a simple procedure to increase the depth of a NN without changing its realization. This can be done by concatenating a NN with another NN which emulates the identity. One possible construction of a NN, the realization of which is the identity, is presented below.
\begin{proposition}[DNN emulation of $\mathrm{Id}$ \cite{PetV2018OptApproxReLU}]\label{prop:Id}
For every $d,L\in \N$ there exists a NN 
$\Phi^{\mathrm{Id}}_{d,L}$ with $L(\Phi^{\mathrm{Id}}_{d,L}) = L$, 
$W(\Phi^{\mathrm{Id}}_{d,L}) \leq 2 d L$, 
$\sizefirst(\Phi^{\mathrm{Id}}_{d,L})\leq 2$ and 
$\sizelast(\Phi^{\mathrm{Id}}_{d,L})\leq 2$ 
such that 
$\mathrm{R} (\Phi^{\mathrm{Id}}_{d,L}) = \mathrm{Id}_{\R^d}$.
\end{proposition}

Occasionally, it will be necessary to parallelize NNs without shared inputs. We call the associated operation the full parallelization. 

\begin{proposition}[Full parallelization of NNs with distinct inputs \cite{EGJS2018}] 
\label{prop:parallSep}
Let $L \in \N$ and let
$$
\Phi^1 = \left(\left(A_1^1,b_1^1\right), \dots, \left(A_{L}^1,b_{L}^1\right)\right), 
\quad 
\Phi^2 = \left(\left(A_1^2,b_1^2\right), \dots, \left(A_{L}^2,b_{L}^2\right)\right)
$$
be two NNs with $L$ layers each and with 
input dimensions $N^1_0=d_1$ and $N^2_0=d_2$, respectively. 

Then there exists a NN, denoted by $\mathrm{FP}(\Phi^1, \Phi^2)$, 
with $(d_1+d_2)$-dimensional input and $L$ layers, 
which we call the \emph{full parallelization of $\Phi^1$ and $\Phi^2$}, such that 
\begin{equation*}
\mathrm{R}\left(\mathrm{FP}\left(\Phi^1,\Phi^2\right)\right) (x_1,x_2) 
= 
  \left(\mathrm{R}\left(\Phi^1\right)(x_1), \mathrm{R}\left(\Phi^2\right)(x_2)\right),
\end{equation*}
$W(\mathrm{FP}(\Phi^1, \Phi^2)) = W(\Phi^1) + W(\Phi^2)$, 
$\sizefirst(\mathrm{FP}(\Phi^1, \Phi^2)) = \sizefirst(\Phi^1) + \sizefirst(\Phi^2)$, 
and
$\sizelast(\mathrm{FP}(\Phi^1, \Phi^2)) = \sizelast(\Phi^1) + \sizelast(\Phi^2)$.
\end{proposition}

Finally, we mention the construction of a special NN the realization of which emulates a scalar multiplication. Such a NN was constructed first in \cite{YAROTSKY2017103}. However, the construction of \cite{YAROTSKY2017103} required increasing numbers of layers for decreasing approximation accuracy. In \cite{PetV2018OptApproxReLU} the following construction with a fixed number of layers was introduced.

\begin{lemma}[{\cite[Lemma A.3]{PetV2018OptApproxReLU}}] \label{lem:MultiplicationNetwork}
Let $\theta > 0$ be arbitrary. Then, for every $L \in \N$ with $L > \theta^{-1}$ and each $K \geq 1$, there are constants $c = c(L,K,\theta) \in \N$, $s = s(K) \in \N$,
and an absolute constant $c' \in \N$ with the following property:
For each $\epsilon \in (0, 1/2)$, there is a NN 
$\widetilde{\times}$ with $L(\widetilde{\times}) \leq c' L$, $W(\widetilde{\times})\leq c \epsilon^{-\theta}$, and such that 
$\widetilde{\times}$ satisfies, for all $x,y \in [-K, K]$,
\begin{itemize}
    \item $\left|xy - \Realization \left(\widetilde{\times}\right)(x,y)\right| \leq \epsilon$,
    \item $\mathrm{R}_{\varrho} \left(\widetilde{\times}\right)(x,y) = 0$ if $x y = 0$.
\end{itemize}
\end{lemma}

\subsection{Robust approximation of Lipschitz functions on manifolds by ReLU networks}
Let $D,q, d \in \N$, and $\mani \subset \R^D$ be a smooth  $d-$dimensional submanifold of $\R^D$ and let $f \in L^\infty(\mani; \R^q)$ be Lipschitz continuous. 
Then, for given $\epsilon >0$, we want to find a NN $\Phi^f$ such that 
$$
    \left\|f - \mathrm{R}\left(\Phi^f\right)\right\|_{L^{\infty}(\mani; \R^q)} \leq \epsilon.
$$
Moreover, we would like to have control over the number of weights $W(\Phi^f)$ and the number of layers $L(\Phi^f)$. In this context, we hope to find that the size of $\Phi^f$ depends on $\epsilon$ and $d$ but not or only weakly on $D$.

One main ingredient in our approximation result for $f$ is the following result for approximation of Lipschitz regular functions on cubes. It is an immediate consequence of {\cite[Corollary 5.1]{he2018relu}}. 
\begin{theorem}\label{thm:LipschitzFunctionApprox}
Let $d,q \in \N$, $K>0$ and let $\Omega = [-K,K]^d$. Let $g: \Omega \to \R^q$ be Lipschitz continuous with Lipschitz constant $C>0$. Then, for every $\epsilon \in (0,1)$, there exists a NN $\Phi^{g,\epsilon}$ such that 
\begin{itemize}
    \item $\left\| g - \mathrm{R}\left(\Phi^{g,\epsilon}\right)\right\|_{L^\infty(\Omega; \R^q)} \leq \epsilon$,
    \item $W\left(\Phi^{g,\epsilon}\right) = \mathcal{O}\left(q \epsilon^{-d}\right)$ for $\epsilon \to 0$,
    \item $L\left(\Phi^{g,\epsilon}\right) = \left \lceil\log_2\left(d+1\right)\right \rceil + 1$.
\end{itemize}
In the estimate above, the implicit constant depends on $C$, $K$, and $d$.
\end{theorem}

Before we prove the main result of this paper, we introduce the property of $(M, \delta)$-coveredness below. 
First, we fix some notation: For a given $y \in \mani$ and $\delta'>0$, we denote by $B_{\delta'}(y)$ the ball with radius $\delta'$ and center at $y$, and by $\mathrm{P}_{y,\mani}$ the orthogonal projection onto the tangent space $T_y\mani$ of $\mani$ at $y$ defined by
$\mathrm{P}_{y,\mani}: B_{\delta'}(y) \cap \mani \to V_y \subset T_y\mani$. 
\begin{definition}\label{def:MDelta}
Let $n \in \mathbb{N}$  and $\mani \subset \mathbb{R}^n$ be a relatively compact manifold.  We say that $\mani$ is $(M, \delta)$-\emph{covered}, where $M \in \mathbb{N}$ and $\delta>0$ if there exist a set of points $(y_i)_{i=1}^M \subset \mani$ such that
$$
    \bigcup_{i = 1}^M B_{\frac{\delta}{4}}(y_i) \supset \mani,
$$
where the inverse of $\mathrm{P}_{y_i,\mani}$ is Lipschitz continuous with Lipschitz constant bounded by $2$ for all $i \in \{1, \dots, M\}$.
\end{definition}

As previously stated, Definition \ref{def:MDelta} is crucial for the proof of Theorem \ref{thm:LipschitzFunctionOnManifoldBetterEstimate2}. 
In order to reconstruct a stable approximation for the inverse operator of every inverse problem presented in Section \ref{sec:admissibleInverseProblems}, we need to prove that all the corresponding manifolds $\mathcal{M}$ are $(M, \delta)$-covered for some values of $M$ and $\delta$. To this end, notice that if $F$ is a bi-Lipschitz diffeomorphism, then for $(a,b)^d \subset \mathbb{R}^d$ with $a,b \in \mathbb{R}$, it follows that 
$$
    \widetilde{F }((a,b)^d) \subset \overline{\widetilde{F}((a,b)^d)} \subset
\widetilde{F }((a-\eps,b + \eps)^d),
$$
for $\eps >0$. Hence, $\mathcal{M} = \widetilde{F }((a,b)^d)$ is a relatively compact subset of a
smooth submanifold. The fact that $\mathcal{M}$ and $\widetilde{F }((a-\eps,b + \eps)^d)$ are submanifolds follows from \cite[Proposition 5.2]{lee2013smooth}. The $(M, \delta)$-coveredness then follows from a standard compactness argument.

Using Theorem \ref{thm:LipschitzFunctionApprox}, we will demonstrate in Theorem \ref{thm:LipschitzFunctionOnManifoldBetterEstimate2} below, that realizations of relatively small NNs can well approximate every Lipschitz continuous function $f\in L^\infty(\mani; \R^q)$. 
In addition, we describe how robust the realizations of the approximating NNs are to noise in the ambient space. By robustness, we mean that the realizations are Lipschitz regular functions that dampen the noise.
We will see that, in this regard, a large ambient dimension $D$ compared to the intrinsic dimension $d$ of the manifold will make the realization of the NN \emph{more} robust to ambient noise. 
We state the result below. In this theorem, the distance on the manifold $\mani$ is the restriction of the Euclidean distance from the ambient space. 

\begin{theorem}\label{thm:LipschitzFunctionOnManifoldBetterEstimate2}
Let $M$, $D$, $d$, $q \in \mathbb{N}$, $\delta>0$ and let $\mani \subset [0,1]^D$ be an $(M,\delta)$-covered d-dimensional submanifold.

Then, for every Lipschitz continuous function $f: \mani \to [0,1]^q$ with Lipschitz constant $C>0$ and every $\epsilon \in (0,1)$ there is a NN $\Phi^{f, \epsilon}_{\mani}$ with $D-$dimensional input such that the following holds
\begin{itemize}
    \item $\left\| f - \mathrm{R}\left(\Phi^{f, \epsilon}_{\mani}\right)\right\|_{L^\infty(\mani; \R^q)} \leq \epsilon$,
    \item $W\left(\Phi^{f, \epsilon}_{\mani}\right) = \mathcal{O}\left( q M^{d+1} \epsilon^{-d}\right)$ for $\epsilon \to 0$,
    \item $L\left(\Phi^{f, \epsilon}_{\mani}\right) = \left \lceil\log_2\left(d\right)\right \rceil + \left \lceil\log_2\left(D\right)\right \rceil + c$,
\end{itemize}
where in the implicit constant in the estimates of the weights above there is an \emph{additive} constant that depends on $\mani$ (hence on D) and $C$, and $c>0$ is a universal constant. Moreover, $\mathrm{R}(\Phi^{f, \epsilon}_{\mani})([0,1]^D) \subset [0,1]^q$.

Additionally, there exists $r_0 > 0$ such that for Gaussian random variables $\eta = (\eta_i)_{i = 1}^D$ where all components are independent, zero-mean, and satisfy $\mathbb{E}(| \eta_i |^2)= \sigma^2 \leq r_0^2$, it holds that
\begin{align}
\mathbb{E}\left(\sup_{x \in \mani}\left|\mathrm{R}\left(\Phi^{f, \epsilon}_{\mani}\right)(x+\eta)\! -\! \mathrm{R}\left(\Phi^{f, \epsilon}_{\mani}\right)(x)\right|^2\right)  \leq &\dfrac{ \hat{C} }{D} \mathbb{E}(|\eta|^2) \cdot \left ( \sqrt{2 d \log(M)} + 2 \log(M) + d  \right) \label{eq:TheProbabilisticStatement}\\
& \qquad + 8 \epsilon^2\! + \! 2De^{-\frac{r_0^2}{2\sigma^2}},\nonumber\\
&= \hat{C} \sigma^2 \cdot \left ( \sqrt{2 d \log(M)} + 2 \log(M) + d  \right) \label{eq:independentOfD}\\
& \qquad + 8 \epsilon^2\! + \! 2De^{-\frac{r_0^2}{2\sigma^2}},\nonumber 
\end{align}

where the implicit constant $\hat{C}>0$ depends on $C$ only.
\end{theorem}
\begin{remark}\label{rem:RotationInvariant}
\begin{enumerate}
\item Theorem \ref{thm:LipschitzFunctionOnManifoldBetterEstimate2} demonstrates that the inverse problem can be approximately solved using a NN of moderate size. Notably, the size of the NN increases asymptotically for $\epsilon \to 0$ and the size and depth depend weakly on $D$ (through additive constants) and strongly on $d$ (in the form of the approximation rate). 
\item The realization of an approximating NN acting on $\R^D$ is robust to small perturbations. Concretely, in \eqref{eq:TheProbabilisticStatement}, the variance of the noise $\eta$ is multiplied with a dampening factor which is inversely propositional to the dimension $D$. More concretely, the noise in the data is multiplied with a dampening factor that decreases when the ratio between the ambient dimension and the manifold dimension increases. 

Note that, in this model, the variance of the noise is $D \sigma^2$. Hence, for $\sigma$ constant, the variance of the noise in the data scales linearly with $D$, while the variance of the noise in the output of the neural network is essentially independent of $D$ as we can see in \eqref{eq:independentOfD}. 

\item We will demonstrate in the proof of Theorem \ref{thm:LipschitzFunctionOnManifoldBetterEstimate2} that $\hat{C} \leq 8C^2$. In practice, one can expect that increasing the number of measurements, i.e. $D$ implies that the Lipschitz constant of the inverse operator will decrease. This suggests that \emph{the stability of the realizations of the NNs will increase with increasing dimension $D$ even beyond the dampening factor $d/D$.} We will observe this phenomenon in the numerical experiments in Section \ref{sec:NumExps}.
\end{enumerate}
\end{remark}

\begin{remark}
The proof of Theorem \ref{thm:LipschitzFunctionOnManifoldBetterEstimate2}, which is given below, is based on a concrete construction of a NN. This construction is based on first applying a partition of unity on $[0,1]^D$ to localize the function $f$ and then locally writing $f$ as smooth functions composed with projections onto low dimensional sets. 
This idea has been used repeatedly, such as, for example, in \cite{chen2019nonparametric, schmidt2019deep, shaham2018provable, nakada2020adaptive}. 
However, a stability statement, such as the one of \eqref{eq:TheProbabilisticStatement}, is not included in any of the references above. 

Since this is our focus, our construction differs from the constructions in \cite{chen2019nonparametric, schmidt2019deep, shaham2018provable, nakada2020adaptive}. 
In particular, we require an exact and very localized partition of unity for which we use results of \cite{he2018relu}, which has not been used in the results above. Moreover, we attempt to produce approximation results where the depth is independent of the accuracy, which is not done in  \cite{chen2019nonparametric, schmidt2019deep}. 
The approximation of the localized functions in \cite{shaham2018provable}, is based on wavelet-type approximation and \cite{chen2019nonparametric, schmidt2019deep, nakada2020adaptive} on Taylor polynomials. Instead, our result uses finite element-based approximation. 
\end{remark}

\begin{proof}
Since $\mani$ is $(M,\delta)$-covered, there exists a set $Y^{\mani} := (y_i)_{i=1}^{M}  \subset \mani$ such that
$$
\bigcup_{i = 1}^{M} B_{\frac{\delta}{4}}(y_i) \supset \mani. 
$$
Moreover, for every $y_i \in Y^{\mani} $, the orthogonal projection onto the tangent space $T_{y_i}\mani$ of $\mani$ at $y_i$ is a diffeomorphism and its inverse, denoted by $\widetilde{P}_{y_i}$, is Lipschitz continuous with Lipschitz constant bounded by $2$.
By the triangle inequality, it follows that
$$
\quad \bigcup_{i = 1}^{M} B_{\frac{\delta}{2}}(y_i) \supset \mani + B_{\frac{\delta}{4}}(0).
$$

Let $G_\delta \coloneqq (\delta/(8 \sqrt{D}))\Z^{D}$ be a uniform grid in $\R^D$ with grid size $\delta/(8\sqrt{D})$. For $z\in G_\delta$, we define $\phi_z^\delta$ as the linear finite element function that equals $1$ on $z$ and vanishes on all other elements of $G_\delta$. It is well-known and not hard to see that $(\phi_z^\delta)_{z \in  G_\delta}$ forms a partition of unity, $\suppp \phi_z^\delta \subset B_{\delta/4}(z)$, and by the boundedness of $\mani$ there exists $Z^{\mani}= (z_{i})_{i = 1}^{M_\delta} \subset G_\delta$ such that
\begin{align}\label{eq:partitionOfUnity}
\sum_{i = 1}^{M_\delta}  \phi_{z_i}^\delta (x) = 1 \text{ for all } x \in \mani + B_{\delta/4}(0)
\end{align}
and $\suppp  \phi_{z_i}^\delta \subset \bigcup_{i = 1}^{M_\delta} B_{3\delta/4}(y_i)$.

By these considerations, it is clear that for each $z\in Z^{\mani}$ there exists $y(z)\in Y^{\mani}$ such that $\suppp \phi_z^\delta \subset B_{\delta}(y(z))$. Therefore, we can write $f$ in the following localized form: for all $x\in \mani$,
\begin{align}
    f(x) = \sum_{i = 1}^{M_\delta} \phi_{z_i}^\delta(x) f(x)
         &= \sum_{i = 1}^{M_\delta} \phi_{z_i}^\delta(x) f\left(\widetilde{P}_{y(z_i)} \left(\mathrm{P}_{y(z_i),\mani}(x)\right)\right)\nonumber\\ 
         &= \sum_{i = 1}^{M} \Bigg(\sum_{\substack{z \in Z^{\mani}\\y(z) = y_i} }\phi_{z}^\delta(x) \Bigg)f\left(\widetilde{P}_{y_i} \left(\mathrm{P}_{y_i,\mani}(x)\right)\right)\nonumber \\
         &\eqqcolon \sum_{i = 1}^{M} \phi_{i}(x) f\left(\widetilde{P}_{y_i} \left(\mathrm{P}_{y_i,\mani}(x)\right)\right).
         \label{eq:partitionOfF}
\end{align}
We define, for $i \in \{1, \dots, M\}$, $\widehat{f}_i \coloneqq f \circ\widetilde{P}_{y_i}$.  
It is not hard to see that $\widehat{f}_i $ is Lipschitz continuous with Lipschitz constant $2C$.
Additionally, we define for $i \in \{1, \dots, M\}$ 
\begin{align*}
\widetilde{f}_i : [0,1] \times V_{y_i} &\to \R^q, \\
(a,b) &\mapsto a \widehat{f}_i(b).
\end{align*}
We have that  
\begin{align}\label{eq:FormOfF}
f(x) = \sum_{i = 1}^{M}\widetilde{f}_i \left(\phi_{i}(x) , \mathrm{P}_{y_i,\mani} (x) \right), \text{ for all } x \in \mani.
\end{align}

Without loss of generality, we can think of $V_{y_i}$ as a compact interval $[-K,K]^D$, since every Lipschitz function can be extended without increasing its Lipschitz constant according to \cite{mcshane1934extension}. By Theorem \ref{thm:LipschitzFunctionApprox}, there exists for every $\epsilon>0$ and every $i  \in \{1, \dots, M  \}$ a NN $\widehat{\Phi}^{i, \epsilon}$ with $d$-dimensional input such that
\begin{itemize}
    \item $\left\|\Realization\left(\widehat{\Phi}^{i, \epsilon}\right) - \widehat{f}_i\right\|_{L^\infty([0,1] \times V_{y_i}; \R^q) }\leq \epsilon$,
    \item $W\left(\widehat{\Phi}^{i, \epsilon}\right) = \mathcal{O}\left(q \epsilon^{-d}\right)$ for $\epsilon \to 0$,
    \item $L\left(\widehat{\Phi}^{i, \epsilon}\right) = \lceil \log_2(d+1)\rceil + 1$.
\end{itemize}
Applying Lemma \ref{lem:MultiplicationNetwork} with $\theta = d$ yields a NN that can be concatenated with $\widehat{\Phi}^{i, \epsilon}$ through Proposition \ref{prop:conc}. This yields that, for every $\epsilon>0$, there exists a NN $\Phi^{\widetilde{f}_i, \epsilon}$ with $(d+1)$-dimensional input such that 
\begin{flalign}
&\quad \ \bullet \ \left|\Realization\left(\Phi^{\widetilde{f}_i, \epsilon}\right)(a,b) - a \Realization\left(\widehat{\Phi}^{i, \epsilon}\right)(b)\right| \leq \epsilon \text{ for all }a \in [0,1], b \in V_{y_i}, &\label{eq:EpsiloNEstimateMultiplication}\\
&\quad \ \bullet \ W\left(\Phi^{\widetilde{f}_i, \epsilon}\right) = \mathcal{O}\left(q\epsilon^{-d}\right) \text{ for  } \epsilon \to 0,&\nonumber\\
&\quad \ \bullet \ L\left(\Phi^{\widetilde{f}_i, \epsilon}\right) = \lceil \log_2(d)\rceil + c,\nonumber&
\end{flalign}
where $c>0$ is a universal constant. Note that the implicit constant in the estimate of the number of weights is independent from $D$.
Moreover, by \cite[Theorem 3.1]{he2018relu} we have that, for every $i\in \{1, \dots, M_\delta\}$ there exists a NN $\Phi_{z_i}^\delta$ with $D$-dimensional input such that  
\begin{itemize}
    \item $\Realization\left(\Phi_{z_i}^\delta\right) = \phi_{i}^\delta$,
    \item $W\left(\Phi_{z_i}^\delta\right) =\mathcal{O}(D)$,
    \item $L\left(\Phi_{z_i}^\delta\right) = \lceil \log_2(D+1)\rceil + 1$.
\end{itemize}
Therefore, we conclude that for every $i\in \{1, \dots, M\}$ there exists a NN $\Phi_{i}$ with $D$-dimensional input such that  
\begin{itemize}
    \item $\Realization\left(\Phi_{i}\right) = \phi_{i}$,
    \item $W\left(\Phi_{i}\right) =\mathcal{O}\left(M_\delta D\right)$,
    \item $L\left(\Phi_{i}\right) = \lceil \log_2(D+1)\rceil + 1$.
\end{itemize}
The estimate on the weights will later only contribute to the overall bound on the weights as an additive constant.

Finally, as $\mathrm{P}_{y_i,\mani}$ is affine linear it can be represented as a one-layer NN $\overline{\Phi}_{P,i,\mani}$ with $D$-dimensional input and $d$-dimensional output. It follows directly from Propositions \ref{prop:conc} and \ref{prop:Id} that $\overline{\Phi}_{P,i,\mani}$ can be extended to have depth $L(\Phi_i)$, i.e., there exists a NN ${\Phi}_{P,i,\mani}$ such that 
\begin{itemize}
    \item $\Realization\left({\Phi}_{P,i,\mani}\right) = \mathrm{P}_{y_i,\mani}$,
    \item $W\left({\Phi}_{P,i,\mani}\right) = \mathcal{O}\left(Dd + d \log_2(D)\right) = \mathcal{O}\left(d D \right)$, 
    \item $L\left({\Phi}_{P,i,\mani}\right) = L\left(\Phi_{i}\right)$.
\end{itemize}
We define for every $i \in \{1, \dots, M\}$
$$
\widetilde{\Phi}^{i, \epsilon} \coloneqq \Phi^{\widetilde{f}_i, \epsilon/{M}} \sconc \Parallel{\Phi_{i}, \Phi_{P,i,\mani}},
$$
and
\begin{align*}
    \Phi^{f, \epsilon}_{\mani} \coloneqq \Phi^{1} \sconc \mathrm{FP}\left(\widetilde{\Phi}^{1, \epsilon}, \widetilde{\Phi}^{2, \epsilon}, \dots, \widetilde{\Phi}^{M, \epsilon}\right), 
\end{align*}
where $\Phi^{1} \coloneqq (( \mathrm{1}_{\R^{1,M}}, 0))$ and $\mathrm{1}_{\R^{1,M}}$ is a row vector of length $M$ with all entries equal to 1. It is clear from the construction of $\Phi^{f, \epsilon}_{\mani}$ and \eqref{eq:FormOfF} that 
\begin{align}
\left\| f - \mathrm{R}\left(\Phi^{f, \epsilon}_{\mani}\right)\right\|_{L^{\infty}(\mani; \R^q)}  &= \left\| \sum_{i = 1}^{M} \widetilde{f}_i \left(\phi_{i}(\cdot) , \mathrm{P}_{y_i,\mani} \cdot \right) - \Realization\left(\widetilde{\Phi}^{i, \epsilon}\right) \right\|_{L^{\infty}(\mani; \R^q)}\\
\label{eq:fandR} &= \left\| \sum_{i = 1}^{M} \widetilde{f}_i \left(\phi_{i}(\cdot) , \mathrm{P}_{y_i,\mani} \cdot \right) - \Realization\left(\Phi^{\widetilde{f}_i, \epsilon/{M}}\right) \left(\phi_{i}(\cdot) , \mathrm{P}_{y_i,\mani} \cdot \right) \right\|_{L^{\infty}(\mani; \R^q)}\leq \epsilon, 
\end{align}
where we applied the triangle inequality to obtain the final estimate. Moreover, by construction, $W(\Phi^{f, \epsilon}_{\mani}) = \mathcal{O}(q  M^{d+1} \epsilon^{-d})  + \mathcal{O}(M \cdot ( M_\delta D +  d D ))$ and $L(\Phi^{f, \epsilon}_{\mani}) = \lceil\log_2(d)\rceil + \lceil\log_2(D)\rceil + c + 3$. 

We assume without loss of generality that $\Phi^{f, \epsilon}_{\mani}$ already satisfies that $\mathrm{R}(\Phi^{f, \epsilon}_{\mani})$ maps into $[0,1]^q$. If this is not the case, then concatenating $\Phi^{f, \epsilon}_{\mani}$ with a simple NN the realization of which implements the function $(x_i)_{i=1}^q \mapsto (\min\{ \max\{0, x_i\}, 1\})_{i=1}^q$ would enforce the desired property for the concatenated NN. This concatenation does not increase the number of weights or the layers by a significant amount.

Finally, we demonstrate statement \eqref{eq:TheProbabilisticStatement}. First of all, set $r_0 \coloneqq \min \{\delta/(2 \sqrt{D}), 1\}$. Then, for any random vector $\eta = (\eta_j)_{j=1}^D$ where each component $\eta_j$ with $j \in [D]$ is a normally distributed random variable with mean zero and $\sigma^2 = \mathbb{E}(| \eta_j |^2)\leq r_0^2$ the following upper bound holds
$$
    \mathbb{P}(|\eta_1| > r_0 \vee  |\eta_2| > r_0 \vee \ldots \vee |\eta_D| > r_0) \leq \sum_{i=1}^D \mathbb{P}(|\eta_i| > r_0) < 2D e^{-r_0^2 / (2 \sigma^2)}  \eqqcolon p.
$$
Therefore, with probability $1- p$, it holds that for all $x \in \mani$
\begin{align*}
\left|\mathrm{R}\left(\Phi^{f, \epsilon}_{\mani}\right)(x) - \mathrm{R}\left(\Phi^{f, \epsilon}_{\mani}\right)(x+\eta)\right| & \leq  \left| \mathrm{R}\left(\Phi^{f, \epsilon}_{\mani}\right)(x) - \sum_{i = 1}^{M}  \phi_{i}(x) \widehat{f}_i\left(\mathrm{P}_{y_i,\mani}(x)\right) \right | \\
& \qquad + \left | \sum_{i = 1}^{M} \phi_{i}(x) \widehat{f}_i\left(\mathrm{P}_{y_i,\mani}(x)\right) - \sum_{i = 1}^{M} \phi_{i}(x+\eta) \widehat{f}_i \left(\mathrm{P}_{y_i,\mani} (x +\eta)\right)\right | \\
& \qquad \qquad + \left | \sum_{i = 1}^{M} \phi_{i}(x+\eta) \widehat{f}_i \left(\mathrm{P}_{y_i,\mani} (x +\eta)\right) - \mathrm{R}\left(\Phi^{f, \epsilon}_{\mani}\right)(x+\eta)     \right|\\
&\leq \left|\sum_{i = 1}^{M} \phi_{i}(x) \widehat{f}_i\left(\mathrm{P}_{y_i,\mani}(x)\right) - \sum_{i = 1}^{M} \phi_{i}(x+\eta) \widehat{f}_i \left(\mathrm{P}_{y_i,\mani} (x +\eta)\right)\right| + 2 \epsilon\\
&\leq \left|\sum_{i = 1}^{M} \phi_{i}(x)\widehat{f}_i\left(\mathrm{P}_{y_i,\mani}(x)\right)  - \phi_{i}(x+\eta)) \widehat{f}_i\left(\mathrm{P}_{y_i,\mani}(x)\right) \right.\\
& \qquad + \left. \sum_{i = 1}^{M} \phi_{i}(x+\eta) \left(\widehat{f}_i \left(\mathrm{P}_{y_i,\mani} (x)\right) - \widehat{f}_i \left(\mathrm{P}_{y_i,\mani} (x +\eta)\right)\right)\right| + 2 \epsilon \\
& \leq \left|\sum_{i = 1}^{M} (\phi_{i}(x)  - \phi_{i}(x+\eta)) f\left(x\right) \right|\\
& \qquad + \left|\sum_{i = 1}^{M} \phi_{i}(x+\eta) \left(\widehat{f}_i\left(\mathrm{P}_{y_i,\mani} (x)\right) - \widehat{f}_i \left(\mathrm{P}_{y_i,\mani} (x +\eta)\right)\right)\right| + 2 \epsilon\\
& \eqqcolon \mathrm{I} + \mathrm{II} +  2 \epsilon,
\end{align*}
where we applied the triangle inequality in the first, \eqref{eq:fandR} in the second, and the definition of $\widehat{f}_i$ in the third step.
From \eqref{eq:partitionOfUnity}, we conclude that 
$$
    \sum_{i = 1}^{M}  \phi_{i} (x+ \eta) = \sum_{i = 1}^{M}  \phi_{i} (x) = 1
$$
and therefore, $\mathrm{I}=0$. Also, due to the linearity of projections and the Lipschitz continuity of $\widehat{f}_i$ 
we have for $\widetilde{C}= 2C$ that 
\begin{align*}
\mathrm{II}    \leq & \hspace{0.1 cm}  \widetilde{C} \sum_{i = 1}^{M} \phi_{i}(x+\eta)| \mathrm{P}_{y_i,\mani} (\eta)|\\
\leq & \hspace{0.1 cm} \widetilde{C} \sum_{i=1}^{M} \phi_{i} (x+\eta ) \cdot \left ( \max_{j \in [M]} |\mathrm{P}_{y_j, \mani} (\eta)| \right )  \\
 = & \hspace{0.1 cm} \widetilde{C} \max_{j \in [M]} |\mathrm{P}_{y_j, \mani} (\eta)|.
\end{align*} 
Thus,
$$
    \sup_{x \in \mani} \left|\mathrm{R}\left(\Phi^{f, \epsilon}_{\mani}\right)(x+\eta) - \mathrm{R}\left(\Phi^{f, \epsilon}_{\mani}\right)(x)\right|^2 = (\mathrm{II} +  \epsilon )^2 \leq  \left( \widetilde{C} \max_{j \in [M]} |\mathrm{P}_{y_j, \mani} (\eta)| +2 \eps \right)^2.
$$
After the estimate $(z+2\epsilon)^2 \leq 2 z^2 + 8\epsilon^2$ is applied, we obtain
\begin{equation}\label{eq:Max1}
    \sup_{x \in \mani} \left|\mathrm{R}\left(\Phi^{f, \epsilon}_{\mani}\right)(x+\eta) - \mathrm{R}\left(\Phi^{f, \epsilon}_{\mani}\right)(x)\right|^2  \leq  2 \left (\widetilde{C} \max_{j \in [M]} |\mathrm{P}_{y_j, \mani} (\eta)|\right )^2 +8 \eps^2.
\end{equation} 
We observe due to the boundedness of $\mathrm{R}(\Phi^{f, \epsilon}_{\mani})$ that
\begin{align}\label{eq:Max1New}
    \mathbb{E}\left(\sup_{x \in \mani} \left|\mathrm{R}\left(\Phi^{f, \epsilon}_{\mani}\right)(x+\eta) - \mathrm{R}\left(\Phi^{f, \epsilon}_{\mani}\right)(x)\right|^2\right) \leq \mathbb{E}\left( 2 \left (\widetilde{C} \max_{j \in [M]} |\mathrm{P}_{y_j, \mani} (\eta)|\right )^2 +8 \eps^2 \right) + p.
\end{align}

Next, we will find a bound for the first term on the right-hand side of the inequality \eqref{eq:Max1New}. As $\eta = (\eta_i)_{i=1}^D$ is normally distributed with zero mean and covariance matrix $\sigma^2 I$ , where $I$ denotes the identity matrix, it follows that $y = (\eta_i / \sigma)_{i=1}^D$ is normally distributed with zero mean and covariance matrix $I$.
It is well-known that every orthogonal projection operator $P$ on a linear subspace $V$ is self-adjoint and idempotent. Let $P$ be an orthogonal projection operator mapping to a $d$-dimensional subspace, then by \cite[Theorem 3.1]{graybill1957idempotent}, it holds that
$$
    \Vert Py \Vert^2 = y^T P^T P y = y^T P^2 y = y^T P y
$$
is $\chi^2_d$ distributed. 
Therefore, it can be concluded that $|\mathrm{P}_{y_j, \mani}(\eta)|^2/\sigma^2$ is a $\chi_d^2$ distributed random variable for $j \in [M]$. Moreover, according to the maximal inequality for $\chi_d^2$ random variables (see \cite[Example 2.7]{boucheron2013concentration}), it holds that
\begin{equation*}
        \mathbb{E} \left(\max_{j \in [M]} \dfrac{|\mathrm{P}_{y_j, \mani}(\eta)|^2}{\sigma^2}  - d \right) \leq \sqrt{2 d \log(M)} + 2 \log(M).
\end{equation*}  
Hence, we observe that
\begin{align}
\mathbb{E} \left( \left(\max_{j \in [M]} |\mathrm{P}_{y_j, \mani} (\eta)| \right )^2 \right ) & = \mathbb{E} \left( \max_{j \in [M]} |\mathrm{P}_{y_j, \mani} (\eta)|^2 \right ) \nonumber\\
& \leq \left( \sqrt{2 d \log(M)} + 2 \log(M) + d \right) \cdot \sigma^2 \nonumber\\
& \leq  \left ( \sqrt{2 d \log(M)} + 2 \log(M) + d \right) \cdot \dfrac{1}{D} \mathbb{E}(|\eta|^2),\label{eq:estimateOfmaximumInProof}
\end{align}
and therefore
\begin{eqnarray*}
\mathbb{E}\left(\sup_{x \in \mani}\left|\mathrm{R}\left(\Phi^{f, \epsilon}_{\mani}\right)(x+\eta) - \mathrm{R}\left(\Phi^{f, \epsilon}_{\mani}\right)(x)\right|^2\right) & \leq & \mathbb{E}(2 \mathrm{II}^2+8\epsilon^2) + p \\
 & \leq &  2 \widetilde{C}^2  \cdot  \left ( \sqrt{2 d \log(M)} + 2 \log(M) + d  \right)\cdot \dfrac{1}{D} \mathbb{E}(|\eta|^2) + 8 \epsilon^2 + p\\
 & \leq &\dfrac{ \hat{C} }{D} \mathbb{E}(|\eta|^2) \cdot \left ( \sqrt{2 d \log(M)} + 2 \log(M) + d  \right) + 8 \epsilon^2 + p,
\end{eqnarray*}
where $\hat{C} = 2 \widetilde{C}^2 = 8 C^2$.
\end{proof}

\begin{remark}\label{rem:boundedweights}
The weights in the NNs used in the construction of Theorem \ref{thm:LipschitzFunctionOnManifoldBetterEstimate2} are all polynomially bounded in $\epsilon$. 
This can be seen by observing that the weights are polynomially bounded for all NNs used as building blocks. For the multiplication NN of Lemma \ref{lem:MultiplicationNetwork} this is explicitly stated in \cite[Lemma A.3]{PetV2018OptApproxReLU}. For the NNs of Theorem \ref{thm:LipschitzFunctionApprox}, this bound is not spelt out in the original reference, but is clear since linear FEM approximation is emulated, where the weights need to grow only like the inverse of the mesh-size. The same holds for the construction of the partition of unity in the proof. All additional operations require universally bounded weights only. In the sequel, we assume that the weights of the NNs of Theorem \ref{thm:LipschitzFunctionOnManifoldBetterEstimate2} are bounded by $\epsilon^{-r}$ for an $r \in \N$.
\end{remark}

\section{Statistical learning of a robust-to-noise neural network}\label{sec:statisticalLearning}

Theorem \ref{thm:LipschitzFunctionOnManifoldBetterEstimate2} guarantees the existence of a NN the realization of which is robust with respect to noisy inputs. In practice, we also need to find this NN or one with similar properties by performing some form of empirical risk minimization over a set of NNs for an appropriately chosen training set.

If we attempt to create a NN with robustness to noise that implements the inverse to a forward operator $F\colon [0,1]^q \to \mani \subset [0,1]^D$, for $q,D \in \N$, then it appears to be natural to train on noisy data, $(F(x_i) + \eta_i, x_i)_{i=1}^m$, where $m \in \N$ and $\eta_i$ are additive noise vectors. Note that, if $F$ is known, then such data can be readily generated.

However, the theoretically-established robust-to-noise NN of Theorem \ref{thm:LipschitzFunctionOnManifoldBetterEstimate2} is one that instead of approximating $F^{-1}$ approximates, for an $r_0 > 0$, the function
\begin{align*}
    f(x) = \sum_{i = 1}^{M} \phi_{i}(x) \widehat{f}_i\left(\mathrm{P}_{y_i,\mani}(x)\right), \text{ for } x \in \mani + B_{r_0}(0),
\end{align*}
where $\hat{f}_i$ are the Lipschitz continuous functions defined under Equation \eqref{eq:partitionOfF}. 
Note that $f$ only necessarily agrees with $F^{-1}$ on $\mani$. 
To learn $f$, it appears more appropriate to train on samples of the form $(y + \eta, f(y + \eta))$ for $y$ following some distribution on $\mani$ and $\eta$ Gaussian. However, while one could in principle compute $f$ for many practical problems, it would be preferable to only access the forward operator $F$. 

We will show below, that training on values $(F(x_i) + \eta_i, x_i)_{i=1}^m$ with $x_i$ drawn according to some distribution on $[0,1]^q$ and $\eta_i$ Gaussian yields a robust NN as well. 
To formulate the result, we first need a generalization result for empirical risk minimization over sets of NNs with multi-dimensional output.

\begin{definition}\label{def:NeuralNetworkSets}
	Let $q, D, L, W \in \N$, and $B > 0$.
	We denote by $\NN(D, q, L, W, B)$ the set of NNs
	with $L$ hidden layers, $D$-dimensional input and $q$-dimensional output, with at most $W$ non-zero weights, which are bounded in absolute value by $B$. 
  Moreover, we set
  \[
    \NN_\ast (D, q, L, W, B)
    := \big\{
         \Phi \in \NN(D, q, L, W, B)
         \quad\colon\quad
         0 \leq \mathrm{R}(\Phi)(x)_i \leq 1 \quad \forall \, x \in [0,1]^q, i = 1, \dots, q
       \big\}
    .
  \]
  We denote the associated sets of realizations by 
  $\RNN(D, q, L, W, B)$ and $\RNN_\ast(D, q, L, W, B)$, respectively.
\end{definition}

Next, we state a bound for the generalization error of empirical risk minimization when multi-dimensional output is used. Similar results may exist in the literature, see e.g. \cite[Theorem 10.1]{AnthonyBartlett} for the univariate case. To keep the manuscript self-contained, we add a proof of the result.
\begin{proposition}\label{prop:highDimGenBound}
Let $m, q, D, L, W \in \N$. 
Let $g \colon [0,1]^D \to [0,1]^q$ and let $\mathcal{D}$ be a distribution on $[0,1]^D$. 

Then, for every $\Phi \in  \NN_\ast (D, q, L, W, B)$ and every $\epsilon >0$
\begin{align}
&\mathbb{P}_{(x_i)_{i=1}^m \sim \mathcal{D}^m}\left(\mathbb{E}_{x\sim \mathcal{D}}|\mathrm{R}(\Phi)(x) - g(x)|^2 - \frac{1}{m}\sum_{i=1}^m|\mathrm{R}(\Phi)(x_i) - g(x_i)|^2 > \epsilon\right) \nonumber\\
&\qquad \leq 2  \left(44 \frac{2q}{\epsilon} L^4 (\lceil B \rceil \max\{q, W\})^{L+1}\right)^{W q} e^{-m \epsilon^2 / 8}.\label{eq:firstEquationOfHighDimGenBound}
\end{align}
In particular, for $p \in (0,1)$ it holds with probability $1-p$ over the choice of $(x_i)_{i=1}^m \sim \mathcal{D}^m$ that 
\begin{align}
    &\mathbb{E}_{x\sim \mathcal{D}}|\mathrm{R}(\Phi)(x) - g(x)|^2\\
    &\qquad \leq  \frac{1}{m}\sum_{i=1}^m|\mathrm{R}(\Phi)(x_i) - g(x_i)|^2\nonumber \\
    &\qquad \qquad +  \sqrt{\frac{8+8 W q (5 + \ln(q) + 4\ln(L) + \ln(1/\epsilon)+ (L+1) (\ln \lceil B \rceil + \ln \max\{q, W\}))}{m} + \frac{8 \ln(1/p)}{m}}.\label{eq:ourawesomeGeneralisationBound}
\end{align}
\end{proposition}
\begin{proof}
We first observe that $\RNN_\ast (D, q, L, W, B) \subset \bigotimes_{j = 1}^q\RNN_\ast (D, 1, L, W, B)$ and hence, by \cite[Lemma 6.1]{grohs2021proof}, that $\RNN_\ast (D, q, L, W, B)$ can be covered by 
$$
    \left( \frac{44}{\epsilon} L^4 (\lceil B \rceil \max\{q, W\})^{L+1}\right)^{W q}
$$ 
balls in $L^{\infty}([0,1]^D, [0,1]^q)$ of radius $\epsilon>0$, where the $L^\infty$ norm of a function $f:[0,1]^D \to [0,1]^q$ is defined as 
\begin{align*}
    \|f\|_\infty \coloneqq \sup_{k = 1 ,\dots, q} \|f_k\|_\infty.
\end{align*}
For $g \colon [0,1]^D \to [0,1]^q$, it follows that the function class
\begin{align*}
    \mathcal{H}_{q, g} \coloneqq \left\{ x \mapsto \sum_{i=1}^q |(h(x))_i - (g(x))_i |^2 \colon h \in \RNN_\ast (D, q, L, W, B)\right\}
\end{align*}
can be covered by $( (44/\epsilon) L^4 (\lceil B \rceil \max\{q, W\})^{L+1})^{W q}$ balls in $L^{\infty}([0,1]^D)$ of radius $2q\epsilon$ since the map 
\begin{align*}
    K_q \colon L^{\infty}([0,1]^D, [0,1]^q) &\to L^{\infty}([0,1]^D)\\
    h &\mapsto K_q(h)= \sum_{i=1}^q |h_i - g_i |^2
\end{align*}
satisfies for all $h^{(1)}, h^{(2)} \in L^{\infty}([0,1]^D, [0,1]^q) $ that
\begin{align}
   \left \|K_q(h^{(1)}) - K_q(h^{(2)})\right\|_\infty &= \left\|\sum_{i=1}^q |h^{(1)}_i - g_i |^2 - \sum_{i=1}^q |h^{(2)}_i - g_i |^2 \right\|_\infty \nonumber\\
    &\leq \sum_{i=1}^q \left\| |h^{(1)}_i - g_i |^2 - |h^{(2)}_i - g_i |^2 \right\|_\infty \leq 2 q\left\|h^{(1)} - h^{(2)}\right\|_\infty.\label{eq:LipschitzContinuityOfK-d}
\end{align}
Here \eqref{eq:LipschitzContinuityOfK-d} holds since for $x >y$ and $x,y,b \in [0,1]$
$$
    (x-b)^2 - (y - b)^2 = \int_{y}^x 2(z-b) dz \leq 2 (x-y).
$$
Applying \cite[Theorem 10.1]{AnthonyBartlett} to $\mathcal{H}_{q, g}$ yields \eqref{eq:firstEquationOfHighDimGenBound}.
Finally, to demonstrate \eqref{eq:ourawesomeGeneralisationBound}, observe that $\ln(2\cdot 44) \leq 5$, set 
$$
    \epsilon =  \sqrt{\frac{8+8 W q (5 + \ln(q) + 4\ln(L) + \ln(1/\epsilon)+ (L+1) (\ln \lceil B \rceil + \ln \max\{q, W\}))}{m} + \frac{8 \ln(1/p)}{m}},
$$ 
and apply \eqref{eq:firstEquationOfHighDimGenBound}.
\end{proof}

Next we show that training a NN of a specific size on data of the form $(F(x_i) + \eta_i, x_i)_{i=1}^m$ with $x_i$ drawn according to some distribution on $\mani$ and $\eta_i$ Gaussian generates a NN that has robustness properties similar to the NN of Theorem \ref{thm:LipschitzFunctionOnManifoldBetterEstimate2}. In particular, the robustness to noise is independent from $D$ and a small value of $d$ dampens the noise.

\begin{theorem}\label{thm:noiseStabilityAfterTraining}
Let $r,q, m, D \in \N$, $\mathcal{D}$ be a distribution on $K \subset [0,1]^q$, let $F \colon K \to \mani \subset [0,1]^D$, be such that $\mani$, $F^{-1}$ satisfy the assumptions of Theorem \ref{thm:LipschitzFunctionOnManifoldBetterEstimate2}, let $c, C, d, M$ be as in Theorem \ref{thm:LipschitzFunctionOnManifoldBetterEstimate2}, and let $r > 0$ be as in Remark \ref{rem:boundedweights}. Let $\eta \in \R^D$ be a Gaussian random variable where all components are independent, zero-mean with variance $\sigma^2$. Let $S = (s_i^{(1)}, s_i^{(2)})_{i=1}^m = (F(x_i) + \eta_i, x_i)_{i=1}^m$ be a sample where $x_i \sim \mathcal{D} $ and $\eta_i \sim \eta$ are independent random variables for $i = 1, \dots, m$. For $\kappa \geq C \cdot ( \sqrt{2 d \log(M)} + 2 \log(M) + d ) \cdot \sigma^2$, we choose
\begin{align*}
    W & \sim q M^{d+1} \kappa^{-d/2}, \quad L \sim \left \lceil\log_2\left(d\right)\right \rceil + \left \lceil\log_2\left(D\right)\right \rceil + c, \quad B \sim \kappa^{-r/2},
\end{align*}
to be the values corresponding to $\epsilon = \sqrt{\kappa}$ in Theorem \ref{thm:LipschitzFunctionOnManifoldBetterEstimate2} and we let
\begin{align}
    \Phi_S \in \argmin_{\Phi \in \NN_\ast (D, q, L, W, B)}\left\{ \frac{1}{m} \sum_{i=1}^m \left|\mathrm{R}(\Phi)\left(s_i^{(1)}\right) - s_i^{(2)}\right|^2\right\}.
\end{align}

Then, with probability $1-2p$ over the choice of $S$
\begin{align}\label{eq:differenTypeOfTraining}
    \mathbb{E}_{x\sim \mathcal{D}, \eta} \left|\mathrm{R}\left(\Phi_S\right)(F(x) + \eta) - x\right|^2 \leq \overline{C} \kappa + c' \sqrt{\frac{q\kappa^{-d/2}\ln(1/\kappa)}{m}} + 4De^{-\frac{r_0^2}{2\sigma^2}} + \sqrt{\frac{16\ln(2/p)}{m}},
\end{align}
where $c'>0$ depends on $\mani$ and $\overline{C}>0$ depends only on $C$ (it is in particular independent from $D$, $d$, $q$, and $M$).
\end{theorem}
\begin{proof}
First, we note that for $m \in \N$ and $p \in (0,1/2)$ by Proposition \ref{prop:highDimGenBound} with probability at least $1-p$ over the draw of $S = (F(x_i) + \eta_i, x_i)_{i=1}^m$ it holds that 
\begin{align}
     &\mathbb{E}_{x\sim \mathcal{D}, \eta} \left|\mathrm{R}\left(\Phi_S\right)(F(x) + \eta) - x\right|^2 \nonumber\\
     &\leq \frac{1}{m}\sum_{i=1}^m\left|\mathrm{R}\left(\Phi_S\right)(F(x_i) + \eta_i) - x_i\right|^2\nonumber\\
    &\qquad +   \sqrt{\frac{8+8 W q (5 + \ln(q) + 4\ln(L) + \ln(1/\epsilon)+ (L+1) (\ln \lceil B \rceil + \ln \max\{q, W\}))}{m} + \frac{8 \ln(1/p)}{m}}. \label{eq:theFirstEstimateInTheProof}
\end{align}
Next, we note that per definition of $\Phi_S$
\begin{align}
    \frac{1}{m}\sum_{i=1}^m|\mathrm{R}(\Phi_S)(F(x_i) + \eta_i) - x_i|^2 &=  \inf_{\Phi \in \NN_\ast (D, q, L, W, B)}\frac{1}{m}\sum_{i=1}^m|\mathrm{R}(\Phi)(F(x_i) + \eta_i) - x_i|^2\nonumber\\
    &\leq \inf_{\Phi \in \NN_\ast (D, q, L, W, B)} \left( \frac{2}{m}\sum_{i=1}^m|\mathrm{R}(\Phi)(F(x_i) + \eta_i) - \mathrm{R}(\Phi)(F(x_i))|^2 \right.\nonumber\\
    &\qquad+ \left.\frac{2}{m}\sum_{i=1}^m|\mathrm{R}(\Phi)(F(x_i)) - x_i|^2\right). \label{eq:WellIfEverythingHasANumberWhyNotThisEquation}
\end{align}
By Theorem \ref{thm:LipschitzFunctionOnManifoldBetterEstimate2}, it holds that $\Phi^{f,\sqrt{\kappa}}_{\mani} \in \NN_\ast (D, q, L, W, B)$ and hence
\begin{align}
    \frac{2}{m}\sum_{i=1}^m\left|\mathrm{R}\left(\Phi^{f,\sqrt{\kappa}}_{\mani}\right)(F(x_i)) - x_i\right|^2 \leq 2\kappa. \label{eq:ThisTooCouldBeVeryUsefulLater}
\end{align}
Next, we denote for $\Phi \in \NN_\ast (D, q, L, W, B)$
\begin{align*}
     e_S(\Phi) &\coloneqq \frac{1}{m}\sum_{i=1}^m|\mathrm{R}(\Phi)(F(x_i) + \eta_i) - \mathrm{R}(\Phi)(F(x_i))|^2,\\
     e_\infty(\Phi) &\coloneqq \mathbb{E}_{x \sim \mathcal{D}, \eta}|\mathrm{R}(\Phi)(F(x) + \eta) - \mathrm{R}(\Phi)(F(x))|^2.
\end{align*}
By Hoeffding's inequality, we have that 
\begin{align}
    \mathbb{P}\left( e_S(\Phi) - e_\infty(\Phi) > \sqrt{\ln(2/p)/m}\right) \leq p.
\end{align}
We conclude that with probability at least $1-p$ by \eqref{eq:WellIfEverythingHasANumberWhyNotThisEquation}
\begin{align}
   \frac{1}{m}\sum_{i=1}^m|\mathrm{R}(\Phi_S)(F(x_i) + \eta_i) - x_i|^2 &\leq 2 e_\infty\left(\Phi^{f,\sqrt{\kappa}}_{\mani}\right) + 2\kappa + 2\sqrt{\frac{\ln(2/p)}{m}}. \label{eq:IBetThisWillBeUsefulLater}
\end{align}
By Theorem \ref{thm:LipschitzFunctionOnManifoldBetterEstimate2}, we can estimate 
\begin{align}
e_\infty\left(\Phi^{f,\sqrt{\kappa}}_{\mani}\right) &\leq \dfrac{ \hat{C} }{D} \mathbb{E}(|\eta|^2) \cdot \left ( \sqrt{2 d \log(M)} + 2 \log(M) + d  \right)\! +\! 8 \kappa \! + \! 2De^{-\frac{r_0^2}{2\sigma^2}}\nonumber\\
&\leq C' \kappa + 2De^{-\frac{r_0^2}{2\sigma^2}}, \label{eq:weMightNeedThis}
\end{align}
for a constant $C' = \hat{C} + 8 >0$. 

Applying \eqref{eq:weMightNeedThis} to \eqref{eq:IBetThisWillBeUsefulLater} yields that with probability at least $1-p$
\begin{align}
    \frac{1}{m}\sum_{i=1}^m|\mathrm{R}(\Phi_S)(F(x_i) + \eta_i) - x_i|^2 \leq C''\kappa + 4De^{-\frac{r_0^2}{2\sigma^2}} + 2\sqrt{\frac{\ln(2/p)}{m}}. \label{eq:almostDoneHere}
\end{align}
Combining \eqref{eq:almostDoneHere} with \eqref{eq:theFirstEstimateInTheProof} and noting that since both estimates hold with probability at least $1-p$ the overall estimate holds with probability at least $1-2p$ completes the proof.
\end{proof}

\section{Numerical experiments} \label{sec:NumExps}
For all inverse problems described in Section \ref{sec:admissibleInverseProblems}, we present numerical examples illustrating how robust approximations to the inverse operators can be learned by neural networks. We will proceed similarly for all cases: 
First, given the forward operator $F :\mathcal{X} \to \mathcal{Y}$ between Banach spaces, a basis $B$ that defines a finite-dimensional space $W$ is chosen and the domain of $F$ will be restricted to this subspace. 
Second, as it was shown in Section \ref{sec:admissibleInverseProblems} for the list of admissible inverse problems, given a compact and convex subset $K$ of $W$ and a sequence of finite-rank operators $(\widehat{Q}_N)_{N \in \mathbb{N}}$ where $\widehat{Q}_N: \mathcal{Y} \to Y_N$ and $Y_N$ are linear subspaces of $Y$, there is a $D \in \mathbb{N}$ such that
$$
    \widehat{F} \coloneqq \widehat{Q}_D \circ F \vert_K 
$$
has a Lipschitz continuous inverse. Third, we identify the space $Y_D$ with a Euclidean space $\R^D$ via an isomorphism $T_D: Y_D \to \R^D$ and $W$ with $\R^d$ through the isomorphism $P: \R^d \to W$. Then, the function
$$
    \widetilde{F} \coloneqq T_D \circ \widehat{F} \circ P \vert_{P^{-1}(\lGamma)}
$$
is Lipschitz continuous. Next, $\widetilde{F}^{-1}: \R^D \supset \mathcal{M} \to \R^d$, where $\mathcal{M} \coloneqq \widetilde{F}(P^{-1}(\lGamma))$ and $\lGamma \subset K$, will be approximated by realizations of NNs for different values of $D$ since in practice, we do not know the exact value of $D$. 
In addition, we will suppose that we know the operator $ T_D \circ \widehat{Q}_D$ for every $D \in \mathbb{N}$ and we will denote it as $Q_D$. For the training of the NN, and for fixed ambient dimension $D$, intrinsic dimension $d$, and noise level $\delta$, a data set of $40.000$ noisy data points $(y_i, x_i)_{i=1}^{40.000}$ is generated. 
To analyze the practical visibility of the bounds identified in Theorems \ref{thm:LipschitzFunctionOnManifoldBetterEstimate2} and \ref{thm:noiseStabilityAfterTraining} every component is perturbed with normally-distributed random noise $\eta = (\eta_i)_{i=1}^{40.000}$. Then $(y_i, x_i)_{i=1}^{40.000} = (\widetilde{F}(x_i) + \eta_i, x_i)_{i=1}^{40.000}$. 
This training set was identified in Theorem \ref{thm:noiseStabilityAfterTraining} as being the right one to generate robust-to-noise NNs. 
When the training stage is completed, each NN is evaluated on a test set of $8.000$ randomly generated data points $(y_i, x_i)_{i=1}^{8.000}$ with the same distribution as the training data. 

The predictions $\tilde{x}_i$ returned by the NN for each $y_i$ will be compared with the elements $x_i$. For each inverse problem and fixed $D$, $d$, and $\delta$, the training and testing process are performed 10 times and the mean squared test errors are presented. We recall that the mean squared test error is 
\begin{equation*}
    MSE = \dfrac{1}{n}\displaystyle \sum_{i=1}^{n} (x_i - \tilde{x}_i )^2.
\end{equation*}
Notice that the expected squared norm of the added noise, scales linearly in $D$ and linear in $\mathbb{E}(|\eta |^2)$. We observe that the parameter $D$ does not negatively affect the test error, even though the norm of the noise grows. This mirrors the prediction of Theorem \ref{thm:LipschitzFunctionOnManifoldBetterEstimate2}, especially Equation \eqref{eq:TheProbabilisticStatement}. Moreover, the behavior of the error in all examples also reflects the observation in Remark \ref{rem:RotationInvariant}.

\subsection{Identification of the transmissivity coefficient}

In this subsection, two examples will be presented. They differ only in the ambient dimension $D$.
Given the forward operator $F: C([0,1]) \supset H_{\lambda} \to C([0,1])$ of \eqref{eq:transmissionForwardProblem}, we set $f= 1$, $c_0,c_1 = 0$. The domain of $F$ will be restricted to the subspace $W_0$ of $C([0,1])$ given as the span of the four inner Lagrangian finite elements $\phi_1, \dots, \phi_4$ on the grid $\{0, 1/6,2/6,3/6,4/6,5/6,1\}$  and the constant 1. For $W \coloneqq \spann\{\phi_1, \dots, \phi_4\}$, $\lambda \coloneqq 0.1$, let $P_{W, \lambda}$ be defined as in Subsection \ref{sec:transmissivityIdentification}. Additionally, $(a,b) = (0,1)$ is chosen. Finally, let $Q_D$ be the sampling operator that samples continuous functions on $D$ equidistant points in $[0,1]$. The inverse of the function 
\begin{align*}
\widetilde{F} \coloneqq Q_D \circ F \circ P_{W, \lambda}: (0,1)^4 \to \R^D
\end{align*}
will be approximated by the realization of a four-layer NN with $100$ neurons per layer on $40.000$ perturbed data points $(y_i,x_i)_{i=1}^{40.000}$. Notice that $\mathcal{M} = \widetilde{F}((0,1)^4)$ is a $(M,\delta)-$covered submanifold and Theorem \ref{thm:LipschitzFunctionOnManifoldBetterEstimate2} can be applied.
The mean squared error for different values of $D$ and noise level $\delta$ is displayed in Figure \ref{fig:Problem1} and the values are collected in Table \ref{tab:Tran1}.

\begin{figure}[htb]
    \centering
    \includegraphics[width = 0.95\textwidth]{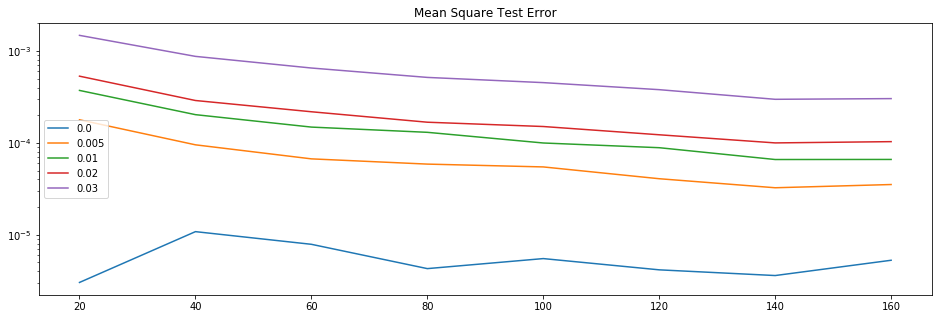}
    \caption{Mean squared errors of the approximation of the solution operator for the transmissivity identification problem on the test set for five different noise levels: $0.00, 0.005, 0.01, 0.02, 0.03$. The $x$-axis shows the ambient dimension which corresponds to the number of measurements $D$, the $y$-axis records the mean squared error of the reconstruction over the test set.}
    \label{fig:Problem1}
\end{figure}

\begin{table}[h]
\begin{center}
\scalebox{0.9}{
\begin{tabular}{ c c c c c c c c c }
Run id & D=20 & D=40 & D=60 & D=80 & D=100 & D=120 & D=140 & D=160 \\
\hline
$\delta = 0.0 $& $ 3.02 \cdot 10^{-5}$ & $1.08\cdot 10^{-5}$ & $7.86 \cdot 10^{-6}$& $4.27 \cdot 10^{-6}$ & $5.49 \cdot 10^{-6}$ & $4.14 \cdot 10^{-6}$ & $3.59\cdot 10^{-6}$ & $5.27 \cdot 10^{-6}    $  \\
\hline
$\delta = 0.005 $ & $1.79 \cdot 10^{-4 }$ & $9.53 \cdot 10^{-5} $ & $6.70 \cdot 10^{-5}$ & $5.88 \cdot 10^{-5}$ & $5.48 \cdot 10^{-5}$ & $4.07 \cdot 10^{-5}$ & $3.25\cdot 10^{-5}$ & $3.52\cdot 10^{-5} $   \\
\hline
$\delta = 0.01 $ & $3.73 \cdot 10^{-4}$ & $2.03 \cdot 10^{-4}$ & $1.49 \cdot 10^{-4}$ & $1.30\cdot 10^{-4}$ & $9.98 \cdot 10^{-5}$ & $8.87 \cdot 10^{-5}$ & $6.59 \cdot 10^{-5}$ & $6.60 \cdot 10^{-5}$\\
\hline
$\delta = 0.02 $ & $7.20 \cdot 10^{-4}$ & $2.90 \cdot 10^{-4}$ & $2.18 \cdot 10^{-4}$ & $1.68 \cdot 10^{-4}$ & $1.51 \cdot 10^{-4}$ & $1.23 \cdot 10^{-4}$ & $1.00 \cdot 10^{-4}$ & $1.03 \cdot 10^{-4}$  \\
\hline
$\delta = 0.03 $ & $1.48 \cdot 10^{-3 }$ & $8.77 \cdot 10^{-4}$ & $6.55 \cdot 10^{-4}$ & $ 5.17 \cdot 10^{-4}$ & $4.54 \cdot 10^{-4}$ & $3.80 \cdot 10^{-4}$ & $2.99 \cdot 10^{-4}$ & $3.00 \cdot 10^{-4}$ \\
\hline

\end{tabular}
}
\end{center}
\caption{Mean squared error of the approximation of the solution operator for the transmissivity identification problem. The intrinsic dimension is $d=4$. }
\label{tab:Tran1}
\end{table}

Next, a second example is discussed where the intrinsic dimension will be altered while the ambient dimension will remain unchanged. Under the same initial condition $c_0=c_1=0$, $f=1$, the operators $P_{W, \lambda}$ and $Q_D$ as previously stated with fixed $D=100$ and the forward operator \eqref{eq:transmissionForwardProblem}, the discretized operator
\begin{align*}
\widetilde{F} \coloneqq Q_{100} \circ F \circ P_{W, \lambda}: (0,1)^d \to \R^{100},
\end{align*}
will be approximated by the realization of a NN for different noise levels and the values $d=4,8,12,16,20$. In this case, the  $d$-dimensional space $W$ is generated by the inner Lagrangian finite elements $(\phi_i)_{i=1}^d$ on the mesh $\{x_0, x_1, \dots , x_{d+1}\}$ where $x_i = {i}/{(d+1)}$ and $i \in \{0, \dots, d+1\}$ for $d=4,8,12,16,20$. In Figure \ref{fig:Problem2}, the mean squared error is shown and the errors are collected in Table \ref{tab:Tra2}. As stated in Equation \eqref{eq:TheProbabilisticStatement}, the error grows as the intrinsic dimension increases. 

\begin{figure}[htb]
    \centering
    \includegraphics[width = 0.95\textwidth]{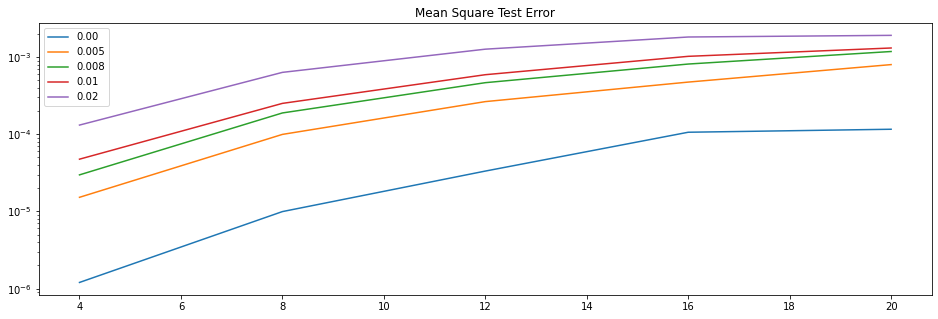}
    \caption{Mean squared errors of the approximation of the solution operator for the transmissivity identification problem on the test set for five different noise levels: $0.00, 0.005, 0.008, 0.01, 0.02$. The $x$-axis shows the intrinsic dimension $d$, the $y$-axis records the mean squared errors of the reconstruction over the test set.}
    \label{fig:Problem2}
\end{figure}

\begin{table}[h]
\begin{center}
\begin{tabular}{ c c c c c c }
Run id & $d=4$ & $d=8$ & $d=12$ & $d=16$ & $d=20$  \\
\hline
$\delta = 0.00 $& $1.20\cdot 10^{-6}$ & $9.93  \cdot 10^{-6}$ & $3.33 \cdot 10^{-5}$ & $1.06  \cdot 10^{-4}$ & $ 1.16\cdot 10^{-4}$  \\
\hline
$\delta = 0.005 $ & $1.52  \cdot 10^{-5}$ & $9.94  \cdot 10^{-5}$ & $ 2.64 \cdot 10^{-4}$ & $4.73  \cdot 10^{-4}$& $ 7.97 \cdot 10^{-4}$\\
\hline
$\delta = 0.008 $ & $ 2.97 \cdot 10^{-5}$& $1.88  \cdot 10^{-4}$& $4.65  \cdot 10^{-4}$ & $ 8.11 \cdot 10^{-4}$& $  1.18\cdot 10^{-2}$\\
\hline
$\delta = 0.01 $ & $4.75  \cdot 10^{-5}$ & $ 2.51 \cdot 10^{-4}$ & $5.90  \cdot 10^{-4}$ & $1.02  \cdot 10^{-3}$ & $ 1.30 \cdot 10^{-3}$ \\
\hline
$\delta = 0.02$ & $1.31  \cdot 10^{-4}$ & $6.33  \cdot 10^{-4}$ & $1.26  \cdot 10^{-3}$ & $1.81  \cdot 10^{-3}$ & $ 1.91 \cdot 10^{-3}$\\
\hline
\end{tabular}
\end{center}
\caption{Mean squared error of the approximation of the solution operator for the transmissivity identification problem. The ambient dimension is $D=100$ and the intrinsic dimension varies.
}
\label{tab:Tra2}
\end{table}

\subsection{Identification of the coefficient in the Euler-Bernoulli equation}

In this subsection, an example of approximation of the inverse operator to the forward operator $F$ of \eqref{eq:EulerBernoulli} with $c_0,c_1,c_2,c_3=0$ and $f=1$ will be discussed. We assume that the coefficient $a$ belongs to the finite-dimensional space $W$ defined as the span of the functions $(a_k)_{k=1}^4$, where
$$
    a_k(x) =  \dfrac{\cos \left( 2 \pi k x \right)}{10}+0.11 \text{ for } x\in [0,1].
$$
Now, we define $P_{W}$ as in Subsection \ref{sec:transmissivityIdentification},  $(a,b)=(0,1)$ and the sampling operator $Q_D$ that samples continuous functions on $D$ equidistant points in $[0,1]$.  Then, the inverse of the function
$$
\widetilde{F} \coloneqq Q_D \circ F \circ P_W : (0,1)^4 \to \R^D
$$
is approximated by the realization of a NN. Notice that $\mathcal{M} = \widetilde{F}((0,1)^4)$ is a $(M,\delta)-$covered submanifold and Theorem \ref{thm:LipschitzFunctionOnManifoldBetterEstimate2} can be applied. Again, a four-layer NN with 100 neurons per layer is trained on a set of 40.000 perturbed data points $(y_i, x_i)_{i=1}^{40.000}$. In Figure \ref{fig:EB}, we present the mean squared error for various values of $D$ and noise levels $\delta$. In Table \ref{tab:EB}, the values of the errors are collected. 

\begin{figure}[htb]
    \centering
    \includegraphics[width = 0.95\textwidth]{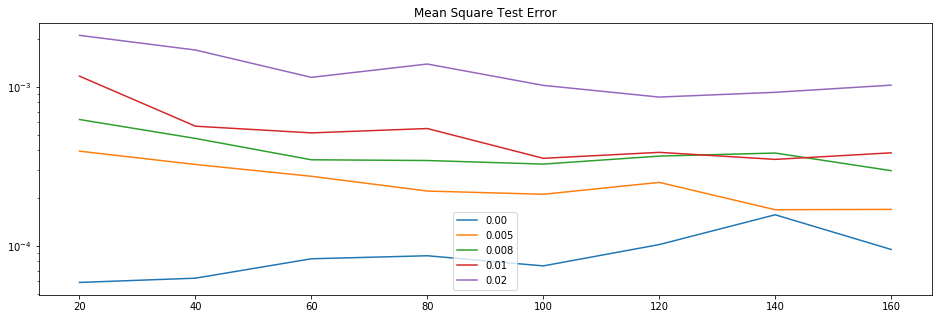}
    \caption{Mean squared errors of the approximation to the inverse problem of the Euler-Bernoulli equation on the test set for five different noise levels: $0.00, 0.005, 0.008, 0.01, 0.02$. The $x$-axis shows the ambient dimension which corresponds to the number of measurements $D$, the $y$-axis records the mean squared error of the reconstruction over the test set.}
    \label{fig:EB}
\end{figure}

\begin{table}[h]
\begin{center}
\scalebox{0.9}{
\begin{tabular}{ c c c c c c c c c }
Run id & D=20 & D=40 & D=60, & D=80 & D=100 & D=120 & D=140 & D=160 \\
\hline
$\delta = 0.00$& $5.90 \cdot 10^{-5 }$ & $6.28 \cdot 10^{-5 }$ &  $8.31 \cdot 10^{-5 }$ &  $8.69 \cdot 10^{-5}$ &  $7.50 \cdot 10^{-5 }$ &  $ 1.02\cdot 10^{-4 }$ &  $1.57 \cdot 10^{-4}$ &  $9.52 \cdot 10^{-5 }$ \\
\hline
$\delta = 0.005 $&  $3.95 \cdot 10^{-4}$ &  $3.25 \cdot 10^{-4}$ &  $2.74 \cdot 10^{-4 }$ &  $2.22 \cdot 10^{-4 }$ &  $2.11 \cdot 10^{-4 }$ &  $2.51 \cdot 10^{-4}$ &  $ 1.69 \cdot 10^{-4}$ &  $1.70 \cdot 10^{-4 }$ \\
\hline
$\delta = 0.008 $&  $6.24 \cdot 10^{-4 }$ & $4.75 \cdot 10^{-4 }$ & $3.48 \cdot 10^{-4 }$ & $3.45 \cdot 10^{-4 }$ & $3.27\cdot 10^{-4 }$ & $3.67 \cdot 10^{-4 }$ & $3.84 \cdot 10^{-4 }$ & $2.98 \cdot 10^{-4 }$ \\
\hline
$\delta = 0.01 $& $1.17 \cdot 10^{-3}$ & $5.66 \cdot 10^{-4 }$ & $ 5.14\cdot 10^{-4 }$ & $ 5.48\cdot 10^{-4 }$ & $3.56 \cdot 10^{-4 }$  & $3.88 \cdot 10^{-4 }$ & $3.50 \cdot 10^{-4 }$ & $3.85 \cdot 10^{-4 }$ \\
\hline
$\delta = 0.02 $ & $2.11 \cdot 10^{-3}$ & $1.71 \cdot 10^{-3}$ & $ 1.15\cdot 10^{-2}$ & $1.39 \cdot 10^{-3}$ & $1.02 \cdot 10^{-3 }$ & $ 8.63\cdot 10^{-4 }$ & $9.25 \cdot 10^{-4 }$ & $1.03 \cdot 10^{-3 }$ \\
\hline

\end{tabular}
}
\end{center}
\caption{Mean squared errors for the inverse problem of the Euler-Bernoulli equation. Here, $d=4$.}
\label{tab:EB}
\end{table}

\subsection{Solution of a non-linear Volterra-Hammerstein integral equation}

For the approximation of the inverse of the non-linear operator \eqref{eq:VolterraHammerstein}, the unknown function $u$ is assumed to belong to the finite-dimensional vector space $W$ generated by the basis functions $(u_k)_{k=1}^d$ such that
$$
    u_{k}(x) = \frac{1}{4}\cos \left( 2 \pi k \cdot \left(x+ \dfrac{1}{4} \right)\right)+1, \text{ for } x \in [0,1].
$$
Once again, we set $(a,b) = (0,1)$, $P_{W,  \lambda}$ as in Section \ref{sec:Volterra} along with $Q_D$ the sampling operator on $D$ equidistant points in $[0,1]$. Then, the inverse of the operator
$$
    \widetilde{F} \coloneqq Q_D \circ F \circ P_{W.\lambda} : (0,1)^d \to \R^D
$$
will be approximated by the realization of a four-layer NN with 100 neurons per layer. Notice that $\mathcal{M} = \widetilde{F}((0,1)^d)$ is a $(M,\delta)-$covered submanifold and Theorem \ref{thm:LipschitzFunctionOnManifoldBetterEstimate2} can be applied. For different values of $\delta$ and $D$ and for $d=4$, the mean squared errors are visualized in Figure \ref{fig:Volt3}. In addition, in Figure \ref{fig:Volt4} the mean squared errors for $d=8$ are depicted. As previously stated, the increase in the intrinsic dimension $d$ yields a substantial increment in the error.

\begin{figure}[htb]
    \centering
    \includegraphics[width = 0.95\textwidth]{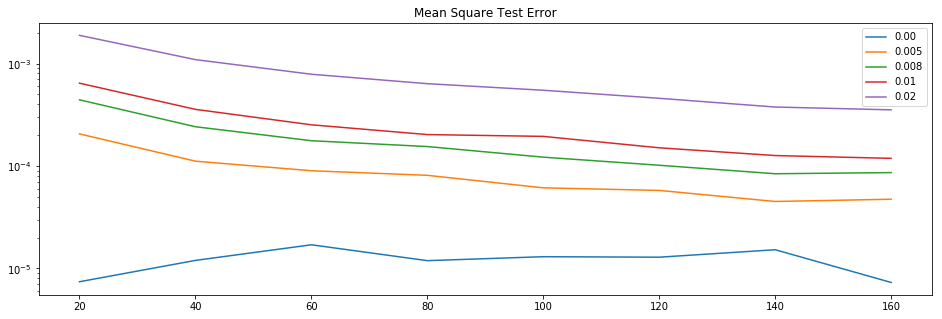}
    \caption{Mean squared error of the approximation to the inverse problem of the Volterra-Hammerstein integral equation on the test set for five different noise levels:  $0.00, 0.005, 0.008, 0.01, 0.02$ and intrinsic dimension $d=4$. The $x$-axis shows the ambient dimension which corresponds to the number of measurements $D$, the $y$-axis records the mean squared error of the reconstruction over the test set.}
    \label{fig:Volt3}
\end{figure}

\begin{figure}[htb]
    \centering
    \includegraphics[width = 0.95\textwidth]{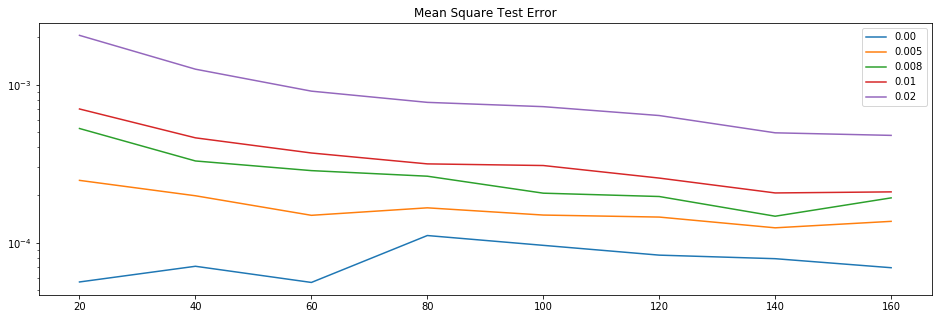}
    \caption{Mean squared errors of the approximation to the inverse problem of the Volterra-Hammerstein integral equation on the test set for five different noise levels: $0.00, 0.005, 0.008, 0.01, 0.02$ and intrinsic dimension $d=8$. The $x$-axis shows the ambient dimension which corresponds to the number of measurements $D$, the $y$-axis records the mean squared error of the reconstruction over the test set.}
    \label{fig:Volt4}
\end{figure}

\begin{table}[h]
\begin{center}
\scalebox{0.9}{
\begin{tabular}{ c c c c c c c c c }
Run id & D=20 & D=40 & D=60, & D=80 & D=100 & D=120 & D=140 & D=160 \\
\hline
$\delta = 0.00 $& $5.46 \cdot 10^{-6}$& $4.39 \cdot 10^{-6}$& $4.72\cdot 10^{-6}$& $5.85\cdot 10^{-6}$& $5.80\cdot 10^{-6}$& $4.69\cdot 10^{-6}$& $5.72 \cdot 10^{-6}$& $6.73 \cdot 10^{-6}$\\
\hline
$\delta = 0.005$ & $2.23\cdot 10^{-4}$& $1.05\cdot 10^{-4}$& $7.88\cdot 10^{-5}$& $7.52\cdot 10^{-5}$& $5.50\cdot 10^{-5}$& $4.50\cdot 10^{-5}$& $3.64\cdot 10^{-5}$& $3.65\cdot 10^{-5}$\\
\hline
$\delta = 0.008$& $6.14 \cdot 10^{-4}$& $3.35 \cdot 10^{-4}$ & $2.51\cdot 10^{-4}$& $1.93\cdot 10^{-4}$& $1.88\cdot 10^{-4}$& $1.37 \cdot 10^{-4}$& $1.11 \cdot 10^{-4}$& $1.10 \cdot 10^{-4}$\\
\hline
$\delta = 0.01$& $1.19 \cdot 10^{-3}$& $6.63\cdot 10^{-4}$& $4.85\cdot 10^{-4}$& $3.77\cdot 10^{-4}$& $3.30 \cdot 10^{-4}$& $2.74 \cdot 10^{-4}$& $2.20\cdot 10^{-4}$& $2.13\cdot 10^{-4}$\\
\hline
$\delta = 0.02$ & $1.86\cdot 10^{-3}$& $1.06\cdot 10^{-3}$& $7.74\cdot 10^{-4}$& $6.23\cdot 10^{-4}$& $5.13 \cdot 10^{-4}$& $4.35\cdot 10^{-4}$& $3.50\cdot 10^{-4}$& $3.38\cdot 10^{-4}$\\
\hline
\end{tabular}
}
\end{center}
\caption{Mean squared error for the inverse problem of the Volterra-Hammerstein equation for $d=4$. }
\label{tab:energy1}
\end{table}

\begin{table}[h]
\begin{center}
\scalebox{0.9}{
\begin{tabular}{ c c c c c c c c c }
Run id & D=20 & D=40 & D=60, & D=80 & D=100 & D=120 & D=140 & D=160 \\
\hline
$\delta = 0.00 $& $5.62 \cdot 10^{-5}$& $7.07 \cdot 10^{-5}$ & $5.58 \cdot 10^{-5}$ & $1.11 \cdot 10^{-5}$ & $9.60 \cdot 10^{-5}$ & $8.32 \cdot 10^{-5}$ & $ 7.90\cdot 10^{-5}$ & $ 6.93 \cdot 10^{-5}$\\
\hline
$\delta = 0.005$ & $2.48 \cdot 10^{-4}$ & $1.98 \cdot 10^{-4}$ & $1.49 \cdot 10^{-4}$& $1.66 \cdot 10^{-4}$& $1.49 \cdot 10^{-4}$& $1.45 \cdot 10^{-4}$& $1.23 \cdot 10^{-4}$& $1.36 \cdot 10^{-4}$\\
\hline
$\delta = 0.008$ & $5.28 \cdot 10^{-5}$ & $3.29 \cdot 10^{-4}$& $2.85 \cdot 10^{-4}$& $2.63 \cdot 10^{-4}$& $ 2.05\cdot 10^{-4}$& $1.96 \cdot 10^{-4}$& $1.47\cdot 10^{-4}$& $1.92 \cdot 10^{-4}$\\
\hline
$\delta = 0.01$ & $ 7.01 \cdot 10^{-4}$ & $4.60 \cdot 10^{-4}$& $3.69 \cdot 10^{-4}$& $3.15 \cdot 10^{-4}$& $3.07 \cdot 10^{-4}$& $ 2.56\cdot 10^{-4}$& $ 2.06 \cdot 10^{-4}$& $2.09 \cdot 10^{-4}$\\
\hline
$\delta = 0.02$ & $2.05 \cdot 10^{-3}$ & $ 1.25 \cdot 10^{-3}$ & $9.10 \cdot 10^{-4}$ & $ 7.72\cdot 10^{-4}$ &$ 7.25 \cdot 10^{-4}$ & $6.37 \cdot 10^{-4}$ & $4.95 \cdot 10^{-4}$ & $4.77 \cdot 10^{-4}$\\
\hline
\end{tabular}
}
\end{center}
\caption{Mean squared errors for the inverse problem of the Volterra-Hammerstein equation for $d=8$. }
\label{tab:Energy2}
\end{table}

\subsection{Solution of the linear inverse gravimetric problem}
\label{sec:NumGravi}
For the solution of the linear inverse gravimetric problem with forward operator $F$ given by \eqref{eq:gravimetric}, the density $\rho$ is assumed to belong to the four-dimensional subspace of $L^2([-1,1]^2)$ generated by the indicator functions:
\begin{align*}
\phi_1(x) &= \mathbbm{1}_{[-0.5,0]\times [0, 0.5]}(x)\\
\phi_2(x) & = \mathbbm{1}_{[0,0.5]\times [0, 0.5]}(x)\\
\phi_3(x)& = \mathbbm{1}_{[-0.5,0]\times [-0.5, 0]}(x)\\
\phi_4(x)& = \mathbbm{1}_{[0, 0.5]\times [-0.5, 0]}(x)
\end{align*}
where $x \in [-1,1]^2$. Again, let $P_{W}: (0,1)^4 \to L^2([-1,1]^2)$ be the operator given by $P_W(\alpha)  = \sum_{k=1}^4 \alpha_k \phi_k$ with $\alpha=(\alpha_k)_{k=1}^4 \in \R^4$. To determine $y_i$, we measure the values of $F(\rho)$ at the following points on the boundary of $[-1, 1]^2$: the first sample is taken at the point $(-1,1) \in [-1,1]^2$ and the successive samples are taken counterclockwise at points with distance $1/D$. We can see that the number of samples is $4D$. Again, a NN is trained to approximate the inverse of the operator
$$
\widetilde{F} \coloneqq Q_{4D} \circ F \circ P_W : (0,1)^4 \to \R^{4D}.
$$
Figure \ref{fig:grav} depicts the mean squared errors for the solution of the problem. The errors are also collected in Table \ref{tab:grav}.

\begin{figure}[ht]
    \centering
    \includegraphics[width = 0.95\textwidth]{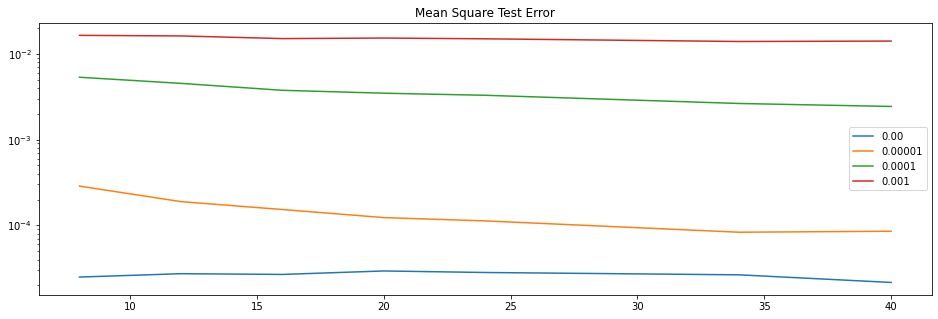}
    \caption{Mean squared errors of the approximation to the inverse problem of the gravimetric problem on the test set for four different noise levels: $0.00, 0.00001, 0.0001, 0.001$ and intrinsic dimension $d=4$. The $x$-axis the value $D$ such that the ambient dimension is $4D$, the $y$-axis records the mean squared errors of the reconstruction over the test set.
    }
    \label{fig:grav}
\end{figure}
\begin{table}[h]
\begin{center}
\begin{tabular}{ c c c c c c c c }
Run id & D=8 & D=12 & D=16, & D=20 & D=24 & D=36 & D=40 \\
\hline
$\delta = 0.00 $& $2.50 \cdot 10^{-5 }$ &  $2.74 \cdot 10^{-5 }$ & $2.68 \cdot 10^{-5 }$ & $ 2.94\cdot 10^{-5}$ & $2.83 \cdot 10^{-5}$ & $2.66 \cdot 10^{-5 }$ & $2.16 \cdot 10^{-5 }$  \\
\hline
$\delta = 10^{-5}$ & $2.87 \cdot 10^{-4 }$ & $1.89 \cdot 10^{-4 }$ & $1.53 \cdot 10^{-4 }$ & $1.23 \cdot 10^{-4 }$ & $1.13 \cdot 10^{-4 }$ & $8.33 \cdot 10^{-5 }$ & $8.22 \cdot 10^{-5 }$ \\
\hline
$\delta = 10^{-4}$  & $5.31 \cdot 10^{- 3}$ & $ 4.5\cdot 10^{-3 }$ & $3.77 \cdot 10^{-3 }$ & $3.46 \cdot 10^{-3}$ & $3.28 \cdot 10^{-3 }$ & $2.63 \cdot 10^{-3 }$ & $2.43 \cdot 10^{-3}$ \\
\hline
$\delta = 10^{-3}$ & $ 1.64 \cdot 10^{-2 }$ & $1.61 \cdot 10^{-2 }$ & $1.50 \cdot 10^{-2 }$ & $1.52 \cdot 10^{-2 }$ & $1.49 \cdot 10^{- 2}$ & $1.39 \cdot 10^{-2}$ & $1.40 \cdot 10^{-2}$ \\
\hline

\end{tabular}
\end{center}
\caption{Mean squared error for the inverse problem of the gravimetric problem. Here, $d=4$. }
\label{tab:grav}
\end{table}

\section*{Acknowledgements}

The authors would like to thank Annalisa Buffa and Pablo Antolin for very inspiring discussions, helpful suggestions, as well as, for introducing the idea of discretizing infinite-dimensional inverse problems to yield finite-dimensional Lipschitz functions which are amenable to treatment by neural networks to the authors.

\bibliographystyle{plain}
\bibliography{references}
\end{document}